\documentclass[reqno,11pt]{article} 
\usepackage[left=2.5cm,right=2.5cm,top=3cm,bottom=3cm,a4paper]{geometry}
\usepackage{amsmath, amssymb}
\usepackage{amsthm, amscd} 
\usepackage[all,cmtip]{xy}
\usepackage{float}
\usepackage{caption}
\usepackage{color}
\usepackage{rotating}

\makeatletter
\DeclareRobustCommand{\rvdots}{%
  \vbox{
    \baselineskip4\p@\lineskiplimit\z@
    \kern-\p@
    \hbox{.}\hbox{.}\hbox{.}
  }}
\makeatother

\newcommand{\Mod}[1]{\ (\textup{mod}\ #1)}
\makeatletter
\def\moverlay{\mathpalette\mov@rlay}
\def\mov@rlay#1#2{\leavevmode\vtop{%
   \baselineskip\z@skip \lineskiplimit-\maxdimen
   \ialign{\hfil$\m@th#1##$\hfil\cr#2\crcr}}}
\newcommand{\charfusion}[3][\mathord]{
    #1{\ifx#1\mathop\vphantom{#2}\fi
        \mathpalette\mov@rlay{#2\cr#3}
      }
    \ifx#1\mathop\expandafter\displaylimits\fi}
\makeatother

\theoremstyle{plain} 
\newtheorem{theorem}{\indent\sc Theorem}[section]
\newtheorem{lemma}[theorem]{\indent\sc Lemma}
\newtheorem{corollary}[theorem]{\indent\sc Corollary}
\newtheorem{proposition}[theorem]{\indent\sc Proposition}

\theoremstyle{definition} 
\newtheorem{definition}[theorem]{\indent\sc Definition}

\newtheorem{remark}[theorem]{\indent\sc Remark}
\newtheorem{example}[theorem]{\indent\sc Example}

\newtheorem{thmx}{Theorem}

\makeatletter
\def\address#1#2{\begingroup
\noindent\parbox[t]{7.8cm}{%
\small{\scshape\ignorespaces#1}\par\vskip1ex
\noindent\small{\itshape E-mail address}%
\/: #2\par\vskip4ex}\hfill%
\endgroup}%
\makeatother

\title{Class fields and form class groups for solving certain quadratic Diophantine equations}
\author{
\textsc{Ho Yun Jung, Ja Kyung Koo, Dong Hwa Shin and Dong Sung Yoon$^*$} 
}
\date{} 
\begin{document}

\allowdisplaybreaks

\maketitle

 \footnote{
\thanks{$^*$Corresponding author.\\
The first named author was supported by the National Research Foundation of Korea (NRF) grant funded by the Korea government (MSIT) (No. RS-2023-00252986).
The third named author was supported
by Hankuk University of Foreign Studies Research Fund of 2024 and by the National Research Foundation of Korea (NRF) grant funded by the Korea government (MSIT) (No. RS-2023-00241953).
}
}

\begin{abstract}
Let $K$ be an imaginary quadratic field and $\mathcal{O}$ be an order in $K$.
We construct class fields associated with form class groups which are isomorphic to
certain $\mathcal{O}$-ideal class groups
in terms of the theory of canonical models due to Shimura. 
As its applications, by using such class fields, for a positive integer $n$ 
we first find primes of the form $x^2+ny^2$ with additional conditions on $x$ and $y$.  
Second, by utilizing these form class groups, we derive a
congruence relation on special values of a modular function of higher level
as an analogue of Kronecker's congruence relation. 
\end{abstract}

\maketitle

\tableofcontents

\section {Introduction}

In the Disquisitiones Arithmeticae published in 1801 Gauss (\cite{Gauss}) introduced
the direct composition on the set of primitive positive definite binary quadratic forms of given (even) discriminant.
After Gauss, there had been several works done to generalize Gauss' composition to an arbitrary ring
(\cite{B-D}, \cite{B-P}, \cite{Kaplansky}, \cite{Kneser}). And in 2004 Bhargava discovered higher
composition laws and gave a simplification of Gauss' composition in view of certain cubes of integers
(\cite{Bhargava}). 
Furthermore, Wood (\cite{Wood}) gave a complete statement of the relationship between binary quadratic forms
and modules for quadratic algebras over any base ring or scheme.
\par
On the other hand, Gauss first paved a new road of using class field in order to solve the famous 
quadratic equation $x^2+y^2=p$ ($p$ a prime) as follows\,:
\begin{equation*}
x^2+y^2=p\quad\Longleftrightarrow\quad
\textrm{$p=2$ or $p$ splits completely in $\mathbb{Q}(\sqrt{-1})$}.
\end{equation*}
Here $\mathbb{Q}(\sqrt{-1})$ is a class field (=\,abelian extension) of $\mathbb{Q}$.
And, by the effort of many people including H. Weber, D. Hilbert and E. Artin
the above equation was further extended to find primes of the form $x^2+ny^2$ ($n$ is a positive integer)
(see \cite{Cox}). 
\par
Let $K$ be an imaginary quadratic field of discriminant $d_K$ with ring of integers
$\mathcal{O}_K$, and $\mathcal{O}$ be an order in $K$. 
We denote by $\ell_\mathcal{O}$ and $D_\mathcal{O}$ the conductor and the discriminant of $\mathcal{O}$,
respectively.  
Let $N$ be a positive integer and $G$ be a subgroup of the unit group $(\mathbb{Z}/N\mathbb{Z})^\times$. 
Let $I(\mathcal{O})$ be the group of proper fractional $\mathcal{O}$-ideals and 
 $P(\mathcal{O})$ be its subgroup of principal fractional $\mathcal{O}$-ideals. 
 Generalizing the $\mathcal{O}$-ideal class group $\mathcal{C}(\mathcal{O})=I(\mathcal{O})/P(\mathcal{O})$, 
we define the group $\mathcal{C}_G(\mathcal{O},\,N)$ by 
\begin{equation*}
\mathcal{C}_G(\mathcal{O},\,N)=I(\mathcal{O},\,N)/P_G(\mathcal{O},\,N)
\end{equation*}
where $I(\mathcal{O},\,N)$ and $P_G(\mathcal{O},\,N)$ are subgroups of $I(\mathcal{O})$ and $P(\mathcal{O})$ given by
\begin{eqnarray*}
I(\mathcal{O},\,N)&=&\langle\mathfrak{a}~|~\textrm{$\mathfrak{a}$ is
a nontrivial proper $\mathcal{O}$-ideal prime to $N$}\rangle,\\
P_G(\mathcal{O},\,N)&=&\langle\nu\mathcal{O}~|~\nu\in\mathcal{O}\setminus\{0\}~\textrm{and}~
\nu\equiv t\Mod{N\mathcal{O}}~\textrm{for some}~t\in\mathbb{Z}~
\textrm{satisfying}~t+N\mathbb{Z}\in G\rangle,
\end{eqnarray*}
respectively. We see that if $N=1$, then $\mathcal{C}_G(\mathcal{O},\,N)$ is nothing but $\mathcal{C}(\mathcal{O})$. 
When $G$ is the trivial subgroup of $(\mathbb{Z}/N\mathbb{Z})^\times$, 
the class field of $K$ associated with $\mathcal{C}_G(\mathcal{O},\,N)$ was first considered by 
H. S\"{o}hngen (\cite{Sohngen}) and later investigated by P. Stevenhagen (\cite{Stevenhagen}). See also \cite{Cho} and 
\cite[$\S$15]{Cox}. 
\par
On the other hand,  let $\mathcal{Q}(D_\mathcal{O},\,N)$ be the set of binary quadratic forms
given by
\begin{equation*}
\mathcal{Q}(D_\mathcal{O},\,N)=\{ax^2+bxy+cy^2\in\mathbb{Z}[x,\,y]~|~
\gcd(a,\,b,\,c)=\gcd(a,\,N)=1,~b^2-4ac=D_\mathcal{O},~a>0\}.
\end{equation*}
The congruence subgroup
\begin{equation*}
\Gamma_G=\left\{\gamma\in\mathrm{SL}_2(\mathbb{Z})~|~
\gamma\equiv\begin{bmatrix}\mathrm{t}&\mathrm{*}\\0&\mathrm{*}\end{bmatrix}
\Mod{NM_2(\mathbb{Z})}~\textrm{for some interger $t$ such that
$t+N\mathbb{Z}\in G$}\right\}
\end{equation*}
of $\mathrm{SL}_2(\mathbb{Z})$ gives rise to an equivalence relation
$\sim_{\Gamma_G}$ on $\mathcal{Q}(D_\mathcal{O},\,N)$ as follows : 
for $Q,\,Q'\in\mathcal{Q}(D_\mathcal{O},\,N)$
\begin{equation*}
Q\sim_{\Gamma_G} Q'~\Longleftrightarrow~
Q'\left(\begin{bmatrix}x\\y\end{bmatrix}\right)
=Q\left(\gamma\begin{bmatrix}x\\y\end{bmatrix}\right)~
\textrm{for some}~\gamma\in\Gamma_G. 
\end{equation*}
We denote the set of equivalence classes by $\mathcal{C}_{\Gamma_G}(D_\mathcal{O},\,N)$, that is,
\begin{equation*}
\mathcal{C}_{\Gamma_G}(D_\mathcal{O},\,N)=\mathcal{Q}(D_\mathcal{O},\,N)/
\sim_{\Gamma_G}.
\end{equation*}
For each $Q(x,\,y)=ax^2+bxy+cy^2\in\mathcal{Q}(D_\mathcal{O},\,N)$, let 
$\omega_Q$ be the zero of the quadratic polynomial $Q(x,\,1)$ lying in the
complex upper half-plane $\mathbb{H}=\{\tau\in\mathbb{C}~|~
\mathrm{Im}(\tau)>0\}$, namely,
\begin{equation}\label{w_Q}
\omega_Q=\frac{-b+\sqrt{D_\mathcal{O}}}{2a}.
\end{equation} 
If $N=1$, then $\Gamma_G=\mathrm{SL}_2(\mathbb{Z})$ and
$\sim_{\Gamma_G}$ is Gauss' proper equivalence. 
It is well known that the Gauss direct composition or the Dirichlet composition 
endows $\mathcal{C}_{\mathrm{SL}_2(\mathbb{Z})}(D_\mathcal{O},\,1)$ with
the group structure so that the mapping
\begin{equation*}
\mathcal{C}_{\mathrm{SL}_2(\mathbb{Z})}(D_\mathcal{O},\,1)\rightarrow\mathcal{C}(\mathcal{O}),\quad[Q]\mapsto[[\omega_Q,\,1]]=[\mathbb{Z}\omega_Q+\mathbb{Z}]
\end{equation*}
becomes an isomorphism (\cite[$\S$3.A and $\S$7.B]{Cox}). 
Furthermore, Chen and Yui showed that if  $\mathcal{O}=\mathcal{O}_K$, $G=(\mathbb{Z}/N\mathbb{Z})^\times$, and so 
$D_\mathcal{O}=d_K$, $\Gamma_G=\Gamma_0(N)$, then the mapping
\begin{equation*}
\mathcal{C}_{\Gamma_0(N)}(d_K,\,N)\rightarrow I(\mathcal{O}_K,\,N)/P_{(\mathbb{Z}/N\mathbb{Z})^\times}(\mathcal{O}_K,\,N)~
(\simeq\mathcal{C}(\mathcal{O})),\quad[Q]\mapsto[[\omega_Q,\,1]]
\end{equation*}
is bijective (\cite[Prposition 4.1 and Theorem 4.4]{C-Y}). 
In this paper, generalizing the above special cases, we shall prove the following theorem in
terms of the theory of canonical models for modular curves due to G. Shimura. 

\begin{thmx}[Theorem \ref{Gformclassgroup}]
The map
\begin{eqnarray*}
\psi_{\Gamma_G,\,P_G(\mathcal{O},\,N)}~:~
\mathcal{C}_{\Gamma_G}(D_\mathcal{O},\,N) &\rightarrow&  I(\mathcal{O},\,N)/P_G(\mathcal{O},\,N)\\
\mathrm{[}Q] &\mapsto& [[\omega_Q,\,1]]
\end{eqnarray*}
is a well-defined bijection, and hence the group $\Gamma_G$ induces a form class group of level $N$
\textup{(}in the sense of \textup{Definition \ref{induce}}\textup{)}. 
\end{thmx}

Here, we notice that the case where $G=\{1+N\mathbb{Z}\}$ was dealt with
in authors' recent paper \cite{J-K-S-Y23}.
\par
Since $\mathcal{C}_G(\mathcal{O},\,N)$ is isomorphic to a generalized ideal class group of $K$ modulo $\ell_\mathcal{O}N\mathcal{O}_K$
(Corollary \ref{CK}), there exists a unique abelian extension $K_{\mathcal{O},\,G}$ of $K$
in which every ramified prime divides $\ell_\mathcal{O}N\mathcal{O}_K$
and the generalized ideal class group is isomorphic to $\mathrm{Gal}(K_{\mathcal{O},\,G}/K)$ 
via the Artin map for the modulus $\ell_\mathcal{O}N\mathcal{O}_K$. 
On the other hand, as is well known, the cyclotomic field
$\mathbb{Q}(\zeta_N)$ with $\zeta_N=e^{2\pi\mathrm{i}/N}$ is a Galois extension of $\mathbb{Q}$ whose 
Galois group is isomorphic to $(\mathbb{Z}/N\mathbb{Z})^\times$.
Let $k_G$ be the fixed field of $\mathbb{Q}(\zeta_N)$ by $\det(G^2)$,
and let 
$\mathcal{F}_{\Gamma_G,\,k_G}$ be the field of meromorphic modular functions for $\Gamma_G$
whose Fourier coefficients belong to $k_G$. 
Here, we mean $G^2=\{g^2\,|\,g\in G\}$.
\par
Define the element $\tau_\mathcal{O}$ of $\mathbb{H}$ by
\begin{equation}\label{tauO}
\tau_\mathcal{O}=\left\{\begin{array}{cl}
\displaystyle\frac{-1+\sqrt{D_\mathcal{O}}}{2} & \textrm{if}~D_\mathcal{O}\equiv1\Mod{4},\vspace{0.1cm}\\
\displaystyle\frac{\sqrt{D_\mathcal{O}}}{2} & \textrm{if}~D_\mathcal{O}\equiv0\Mod{4}.
\end{array}\right. 
\end{equation}
Then we shall describe $\mathcal{C}_{\Gamma_G}(D_\mathcal{O},\,N)$
as the Galois group $\mathrm{Gal}(K_{\mathcal{O},\,G}/K)$ 
in an explicit way. 

\begin{thmx}[Corollary \ref{phiOG}]
We have $K_{\mathcal{O},\,G}=K\left(h(\tau_\mathcal{O})~|~h\in\mathcal{F}_{\Gamma_G,\,k_G}~\textrm{is finite at}~\tau_\mathcal{O}\right)$ 
and we get an isomorphism
\begin{eqnarray*}
\phi_{\mathcal{O},\,G}~:~\mathcal{C}_{\Gamma_G}(D_\mathcal{O},\,N)&\stackrel{\sim}{\rightarrow}&\mathrm{Gal}(K_{\mathcal{O},\,G}/K)\\
\mathrm{[}Q] & \mapsto & \left(
h(\tau_\mathcal{O})\mapsto h^{\left[\begin{smallmatrix}
1 & -a'(\frac{b+b_\mathcal{O}}{2})\\
0& a'\end{smallmatrix}\right]}(-\overline{\omega}_Q)~|~h\in\mathcal{F}_{\Gamma_G,\,k_G}~\textrm{is finite at}~\tau_\mathcal{O}\right)
\end{eqnarray*}
where $Q=ax^2+bxy+cy^2\in\mathcal{Q}(D_\mathcal{O},\,N)$ and $a'$ is an integer such that $aa'\equiv1\Mod{N}$. 
\end{thmx}

Now, let $n$ be a positive integer. 
In 2010\, B. Cho (\cite{Cho}) first studied
primes of the form $x^2+ny^2$ with additional conditions on $x$ and $y$, precisely
$x\equiv1\Mod{N}$ and $y\equiv0\Mod{N}$ in view of the class field theory. 
We shall further extend his result as follows.
 
\begin{thmx}[Lemma \ref{equivalence}, Theorem \ref{x^2+ny^2}]
Given a positive integer $n$, there is a monic irreducible polynomial $f(X)\in\mathbb{Z}[X]$ for which 
if $p$ is a prime dividing neither $2nN$ nor the discriminant of $f(X)$, then 
\begin{eqnarray*}
&&\textrm{$p=x^2+ny^2$ for some $x,\,y\in\mathbb{Z}$
such that $x+N\mathbb{Z}\in G$ and $y\equiv0\Mod{N}$}\\
&\Longleftrightarrow&\textrm{$p$ splits completely in $K_{\mathcal{O},\,G}$}\\
&\Longleftrightarrow& 
\left(\frac{-n}{p}\right)=1~\textrm{and $f(X)\equiv0\Mod{p}$ has an integer solution. }
\end{eqnarray*}
\end{thmx}

Lastly, for a lattice $L$ in $\mathbb{C}$, let $j(L)$ denote the
invariant of an elliptic curve isomorphic to $\mathbb{C}/L$. 
To establish the first main theorem of complex multiplication, Hasse (\cite{Hasse}) showed that
for all but a finite number of primes $p$ which are decomposed with respect to $\mathcal{O}$ as
$p\mathcal{O}=\mathfrak{p}\overline{\mathfrak{p}}$, the congruence 
\begin{equation}\label{Kronecker}
j(\mathfrak{p}^{-1}\mathfrak{a})\equiv j(\mathfrak{a})^p\Mod{\mathfrak{P}}
\end{equation}
holds for any proper fractional $\mathcal{O}$-ideal $\mathfrak{a}$ and
any prime $\mathfrak{P}$ of the ring class field $H_\mathcal{O}$
of order $\mathcal{O}$ lying above $\mathfrak{p}\mathcal{O}_K$. 
The congruence (\ref{Kronecker}) 
is called Kronecker's (or, Hasse's) congruence relation. 
As an analogue of (\ref{Kronecker}), we shall derive a congruence relation on 
special values of a modular function of higher level by utilizing the form class group $\mathcal{C}_{\Gamma_G}(D_\mathcal{O},\,N)$.

\begin{thmx}[Theorem \ref{congruence}]\label{app1}
Let $f$ be a meromorphic modular function for $\Gamma_G$ with rational Fourier coefficients which is
integral over $\mathbb{Z}[j]$. 
If $p$ is a prime such that
\begin{enumerate}
\item[\textup{(i)}]  it is relatively prime to $D_\mathcal{O}N$,
\item[\textup{(ii)}] it is decomposed with respect to $\mathcal{O}$,
\item[\textup{(iii)}] $p+N\mathbb{Z}\in G$ or $-p+N\mathbb{Z}\in G$,
\end{enumerate}
then we have the congruence relation
\begin{equation*}
\left(f(\omega)^p-f\left(\frac{\omega}{p}\right)\right)
\left(f(\omega)-f\left(\frac{\omega}{p}\right)^p\right)\equiv0\Mod{p\mathcal{O}_{K_{\mathcal{O},\,G}}}~
\textrm{with}~\omega=\frac{s+\sqrt{D_\mathcal{O}}}{2}
\end{equation*}
where $s$ is an integer satisfying $s^2\equiv D_\mathcal{O}\Mod{4p}$. 
\end{thmx}

In Theorem \ref{app1}, $j$ stands for the elliptic modular function defined on $\mathbb{H}$.
The special case where $\mathcal{O}=\mathcal{O}_K$ and $G=\{1+N\mathbb{Z}\}$ is investigated in \cite{Yoon}. 

\section {Ideal class groups for orders}

In what follows, we let $K$ be an imaginary quadratic field 
of discriminant $d_K$ ($<0$), $\mathcal{O}_K$ be its ring of integers
and $\mathcal{O}$ be an order in $K$ of conductor $\ell_\mathcal{O}$.
For a nontrivial ideal $\mathfrak{a}$ of $\mathcal{O}$, 
we define its norm by $\mathrm{N}_\mathcal{O}(\mathfrak{a})
=|\mathcal{O}/\mathfrak{a}|$ ($<\infty$). Furthermore, we say that $\mathfrak{a}$ is prime to a positive integer $\ell$ if
$\mathfrak{a}+\ell\mathcal{O}=\mathcal{O}$. 

\begin{lemma}\label{basic}
Let $\mathfrak{a}$ be a nontrivial ideal of $\mathcal{O}$. 
\begin{enumerate}
\item[\textup{(i)}] For a positive integer $\ell$,
$\mathfrak{a}$ is prime to $\ell$ if and only if
$\mathrm{N}_\mathcal{O}(\mathfrak{a})$ is relatively prime to $\ell$. 
\item[\textup{(ii)}] If $\mathfrak{a}=\nu\mathcal{O}$ for some $\nu\in\mathcal{O}\setminus\{0\}$, then 
$\mathrm{N}_\mathcal{O}(\nu\mathcal{O})=\mathrm{N}_{K/\mathbb{Q}}(\nu)$. 
\end{enumerate}
\end{lemma}
\begin{proof}
\begin{enumerate}
\item[(i)] See \cite[Lemma 7.18 (i)]{Cox} except replacing $f$ by $\ell$. 
\item[(ii)] See \cite[Lemma 7.14 (i)]{Cox}.
\end{enumerate}
\end{proof}

Recall the definition of $\tau_\mathcal{O}$ given in (\ref{tauO}). 
If we write $\tau_K=\tau_{\mathcal{O}_K}$ for simplicity, then we see that
\begin{equation*}
\mathcal{O}_K=[\tau_K,\,1]=\mathbb{Z}\tau_K+\mathbb{Z}\quad\textrm{and}\quad
\mathcal{O}=[\ell_\mathcal{O}\tau_K,\,1]=\mathbb{Z}\ell_\mathcal{O}\tau_K+\mathbb{Z}.
\end{equation*}
Throughout this paper, we let $N$ be a positive integer and $G$ be a subgroup of $(\mathbb{Z}/N\mathbb{Z})^\times$. 

\begin{lemma}\label{nulemma}
If $\nu\in K\setminus\{0\}$, then we have
\begin{eqnarray}
&&\left\{\begin{array}{l}
\nu\in\mathcal{O},\\
\textrm{$\nu\mathcal{O}$ is prime to $\ell_\mathcal{O}N$,}\\
\textrm{$\nu\equiv a\Mod{N\mathcal{O}}$ for some integer $a$ such that
$a+N\mathbb{Z}\in G$}
\end{array}\right.\label{nu1}\\
&\Longleftrightarrow&
\left\{\begin{array}{l}
\nu\in\mathcal{O}_K,\\
\textrm{$\nu\mathcal{O}_K$ is prime to $\ell_\mathcal{O}N$,}\\
\textrm{$\nu\equiv b\Mod{\ell_\mathcal{O}N\mathcal{O}_K}$ for some integer $b$ such that
$b+N\mathbb{Z}\in G$}
\end{array}\right.\label{nu2}
\end{eqnarray}
\end{lemma}
\begin{proof}
Assume that (\ref{nu1}) holds. Then $\mathrm{N}_\mathcal{O}(\nu\mathcal{O})$ is relatively prime to 
$\ell_\mathcal{O}N$ by Lemma \ref{basic} (i). Moreover, since
$\mathrm{N}_\mathcal{O}(\nu\mathcal{O})=\mathrm{N}_{K/\mathbb{Q}}(\nu)=\mathrm{N}_{\mathcal{O}_K}(\nu\mathcal{O}_K)$
by Lemma \ref{basic} (ii), $\nu\mathcal{O}_K$ is prime to $\ell_\mathcal{O}N$ again by Lemma \ref{basic} (i).
We get by the fact $\nu-a\in N\mathcal{O}=[N\ell_\mathcal{O}\tau_K,\,N]$ that
\begin{equation*}
\nu=rN\ell_\mathcal{O}\tau_K+sN+a\quad\textrm{for some}~r,\,s\in\mathbb{Z}. 
\end{equation*}
If we let $b=sN+a$, then we obtain that
\begin{equation*}
\nu\equiv b\Mod{\ell_\mathcal{O}N\mathcal{O}_K}\quad\textrm{and}\quad
b+N\mathbb{Z}=a+N\mathbb{Z}\in G. 
\end{equation*}
\par
Conversely, assume that (\ref{nu2}) is satisfied. We deduce by
\begin{equation*}
\nu-b\in\ell_\mathcal{O}N\mathcal{O}_K\subseteq N\mathcal{O}\quad\textrm{and}\quad b\in\mathbb{Z}\subset\mathcal{O}
\end{equation*}
that
\begin{equation*}
\nu\in\mathcal{O}\quad\textrm{and}\quad \nu\equiv b\Mod{N\mathcal{O}}~
\textrm{with}~b+N\mathbb{Z}\in G. 
\end{equation*}
Then it follows from Lemma \ref{basic} (i) that $\mathrm{N}_{\mathcal{O}_K}(\nu\mathcal{O}_K)$ is relatively prime to $\ell_\mathcal{O}N$.
Since $\mathrm{N}_\mathcal{O}(\nu\mathcal{O})=\mathrm{N}_{K/\mathbb{Q}}(\nu)=
\mathrm{N}_{\mathcal{O}_K}(\nu\mathcal{O}_K)$
by Lemma \ref{basic} (ii),
we achieve again by Lemma \ref{basic} (i) that $\nu\mathcal{O}$ is prime to $\ell_\mathcal{O}N$. 
\end{proof}

For positive integers $\ell$ and $m$, we denote by
\begin{eqnarray*}
\mathcal{M}(\mathcal{O},\,\ell)&=&\textrm{the monoid of nontrivial proper $\mathcal{O}$-ideals prime to $\ell$},\\
I(\mathcal{O},\,\ell)&=&\textrm{the subgroup of $I(\mathcal{O})$ generated by the elements of $\mathcal{M}(\mathcal{O},\,\ell)$},\\
P(\mathcal{O},\,\ell)&=&\textrm{the subgroup of $I(\mathcal{O},\,\ell)$ generated by $\nu\mathcal{O}$
for $\nu\in\mathcal{O}\setminus\{0\}$ such that $\nu\mathcal{O}\in\mathcal{M}(\mathcal{O},\,\ell)$},\\
P_G(\mathcal{O},\,\ell,\,m)&=&\textrm{the subgroup of $I(\mathcal{O},\,\ell)$ generated by $\nu\mathcal{O}$
for $\nu\in\mathcal{O}\setminus\{0\}$ such that $\nu\mathcal{O}\in\mathcal{M}(\mathcal{O},\,\ell)$}\\
&&\textrm{and}~\nu\equiv a\Mod{m\mathcal{O}}~\textrm{for some}~a\in\mathbb{Z}
~\textrm{satisfying}~a+N\mathbb{Z}\in G.
\end{eqnarray*} 

Here, we see that $P_G(\mathcal{O},\,N,\,N)=P_G(\mathcal{O},\,N)$. 

\begin{lemma}\label{isomorphism}
The mapping
\begin{equation*}
\mathcal{M}(\mathcal{O},\,\ell_\mathcal{O}N)\rightarrow\mathcal{M}(\mathcal{O}_K,\,\ell_\mathcal{O}N),
\quad\mathfrak{a}\mapsto\mathfrak{a}\mathcal{O}_K
\end{equation*}
is well defined, and uniquely derives an isomorphism
$I(\mathcal{O},\,\ell_\mathcal{O}N)\stackrel{\sim}{\rightarrow}I(\mathcal{O}_K,\,\ell_\mathcal{O}N)$.
\end{lemma}
\begin{proof}
See \cite[Lemma 2.6]{J-K-S-Y23}. 
\end{proof}

\begin{proposition}\label{lNK}
We have the natural isomorphism
\begin{equation*}
I(\mathcal{O},\,\ell_\mathcal{O}N)/
P_G(\mathcal{O},\,\ell_\mathcal{O}N,\,N)\stackrel{\sim}{\rightarrow}I(\mathcal{O}_K,\,\ell_\mathcal{O}N)/
P_G(\mathcal{O}_K,\,\ell_\mathcal{O}N,\,\ell_\mathcal{O}N). 
\end{equation*}
\end{proposition}
\begin{proof}
Let $\psi:I(\mathcal{O},\,\ell_\mathcal{O}N)\stackrel{\sim}{\rightarrow}I(\mathcal{O}_K,\,\ell_\mathcal{O}N)$ be the isomorphism
stated in Lemma \ref{isomorphism}. By Lemma \ref{nulemma}, we obtain
$\psi(P_G(\mathcal{O},\,\ell_\mathcal{O}N,\,N))=P_G(\mathcal{O}_K,\,\ell_\mathcal{O}N,\,\ell_\mathcal{O}N)$. 
Thus we establish the isomorphism
\begin{eqnarray*}
I(\mathcal{O},\,\ell_\mathcal{O}N)/
P_G(\mathcal{O},\,\ell_\mathcal{O}N,\,N)&\stackrel{\sim}{\rightarrow}&I(\mathcal{O}_K,\,\ell_\mathcal{O}N)/
P_G(\mathcal{O}_K,\,\ell_\mathcal{O}N,\,\ell_\mathcal{O}N)\\
\mathrm{[}\mathfrak{a}\mathfrak{b}^{-1}] & \mapsto & [(\mathfrak{a}\mathcal{O}_K)(\mathfrak{b}\mathcal{O}_K)^{-1}]
\end{eqnarray*}
where $\mathfrak{a},\,\mathfrak{b}\in\mathcal{M}(\mathcal{O},\,\ell_\mathcal{O}N)$. 
\end{proof}

\begin{lemma}\label{Ptrivial}
The inclusion $P(\mathcal{O},\,\ell_\mathcal{O}N)\hookrightarrow P(\mathcal{O},\,N)$ induces
an isomorphism
\begin{equation*}
P(\mathcal{O},\,\ell_\mathcal{O}N)/P_{\{1+N\mathbb{Z}\}}(\mathcal{O},\,\ell_\mathcal{O}N,\,N)
\stackrel{\sim}{\rightarrow}
P(\mathcal{O},\,N)/P_{\{1+N\mathbb{Z}\}}(\mathcal{O},\,N).
\end{equation*}
\end{lemma}
\begin{proof}
See \cite[Lemma 15.17 and Exercise 15.10]{Cox}. 
\end{proof}

\begin{lemma}\label{Itrivial}
The inclusion $I(\mathcal{O},\,\ell_\mathcal{O}N)\hookrightarrow I(\mathcal{O},\,N)$ gives an 
isomorphism
\begin{equation*}
I(\mathcal{O},\,\ell_\mathcal{O}N)/P_{\{1+N\mathbb{Z}\}}(\mathcal{O},\,\ell_\mathcal{O}N,\,N)
\stackrel{\sim}{\rightarrow}
I(\mathcal{O},\,N)/P_{\{1+N\mathbb{Z}\}}(\mathcal{O},\,N).
\end{equation*}
\end{lemma}
\begin{proof}
See \cite[Proposition 2.13]{J-K-S-Y23}. 
\end{proof}

\begin{proposition}\label{lNN}
The inclusion $I(\mathcal{O},\,\ell_\mathcal{O}N)\hookrightarrow I(\mathcal{O},\,N)$ leads to
an isomorphism
\begin{equation*}
I(\mathcal{O},\,\ell_\mathcal{O}N)/P_G(\mathcal{O},\,\ell_\mathcal{O}N,\,N)
\stackrel{\sim}{\rightarrow}\mathcal{C}_G(\mathcal{O},\,N)=
I(\mathcal{O},\,N)/P_G(\mathcal{O},\,N).
\end{equation*}
\end{proposition}
\begin{proof}
Let $\rho$ be the isomorphism mentioned in Lemma \ref{Itrivial}, and consider the diagram of homomorphisms in Figure
\ref{homo}.
\begin{figure}[t]
\begin{equation*}
\xymatrixcolsep{6pc}
\xymatrix{
I(\mathcal{O},\,\ell_\mathcal{O}N)/P_{\{1+N\mathbb{Z}\}}(\mathcal{O},\,\ell_\mathcal{O}N,\,N)\ar@{->}[r]^\sim_\rho
\ar@{->>}[dd]_{\textrm{natural}}
 & I(\mathcal{O},\,N)/P_{\{1+N\mathbb{Z}\}}(\mathcal{O},\,N) \ar@{->>}[dd]^{\textrm{natural}} \\\\
I(\mathcal{O},\,\ell_\mathcal{O}N)/P_G(\mathcal{O},\,\ell_\mathcal{O}N,\,N) & 
I(\mathcal{O},\,N)/P_G(\mathcal{O},\,N)
}
\end{equation*}
\caption{Homomorphisms of $\mathcal{O}$-ideal class groups}
\label{homo}
\end{figure}
Since the kernels of the left and right vertical homomorphisms are
$P_G(\mathcal{O},\,\ell_\mathcal{O}N,\,N)/
P_{\{1+N\mathbb{Z}\}}(\mathcal{O},\,\ell_\mathcal{O}N,\,N)$ and
$P_G(\mathcal{O},\,N)/
P_{\{1+N\mathbb{Z}\}}(\mathcal{O},\,N)$, respectively,
it suffices to show that
\begin{equation*}
\rho(P_G(\mathcal{O},\,\ell_\mathcal{O}N,\,N)/
P_{\{1+N\mathbb{Z}\}}(\mathcal{O},\,\ell_\mathcal{O}N,\,N))=
P_G(\mathcal{O},\,N)/
P_{\{1+N\mathbb{Z}\}}(\mathcal{O},\,N).
\end{equation*}
It is obvious that 
\begin{equation*}
\rho(P_G(\mathcal{O},\,\ell_\mathcal{O}N,\,N)/
P_{\{1+N\mathbb{Z}\}}(\mathcal{O},\,\ell_\mathcal{O}N,\,N))\subseteq
P_G(\mathcal{O},\,N)/
P_{\{1+N\mathbb{Z}\}}(\mathcal{O},\,N).
\end{equation*}
For the converse inclusion, let $\nu$ be a nonzero element of $\mathcal{O}$
satisfying that $\nu\mathcal{O}$ is prime to $N$ and $\nu\equiv a\Mod{N\mathcal{O}}$ for some
integer $a$ with
$a+N\mathbb{Z}\in G$.
Take an integer $b$ such that $\gcd(b,\,\ell_\mathcal{O}N)=1$ and $ab\equiv1\Mod{N}$.
We then have $\nu b\equiv ab\equiv1\Mod{N\mathcal{O}}$ and so 
\begin{equation}\label{nub}
\nu b\mathcal{O}\in P_{\{1+N\mathbb{Z}\}}(\mathcal{O},\,N).
\end{equation}
Since $b+N\mathbb{Z}=(a+N\mathbb{Z})^{-1}\in G$, we get that
$b\mathcal{O}\in P_G(\mathcal{O},\,\ell_\mathcal{O}N,\,N)$. 
Hence we find that
\begin{eqnarray*}
\rho((b\mathcal{O})^{-1}P_{\{1+N\mathbb{Z}\}}(\mathcal{O},\,\ell_\mathcal{O}N,\,N))&=&
(b\mathcal{O})^{-1}P_{\{1+N\mathbb{Z}\}}(\mathcal{O},\,N)\\
&=&(\nu\mathcal{O})(\nu b\mathcal{O})^{-1}P_{\{1+N\mathbb{Z}\}}(\mathcal{O},\,N)\\
&=&(\nu\mathcal{O})P_{\{1+N\mathbb{Z}\}}(\mathcal{O},\,N)\quad\textrm{by (\ref{nub})}.
\end{eqnarray*}
This observation implies that
\begin{equation*}
\rho(P_G(\mathcal{O},\,\ell_\mathcal{O}N,\,N)/
P_{\{1+N\mathbb{Z}\}}(\mathcal{O},\,\ell_\mathcal{O}N,\,N))\supseteq
P_G(\mathcal{O},\,N)/
P_{\{1+N\mathbb{Z}\}}(\mathcal{O},\,N),
\end{equation*}
which completes the proof. 
\end{proof}

\begin{corollary}\label{CK}
The group $\mathcal{C}_G(\mathcal{O},\,N)$ is isomorphic to the generalized
ideal class group 
$I(\mathcal{O}_K,\,\ell_\mathcal{O}N)/
P_G(\mathcal{O}_K,\,\ell_\mathcal{O}N,\,\ell_\mathcal{O}N)$
through the mapping
sending $[\mathfrak{a}]$ to $[\mathfrak{a}\mathcal{O}_K]$
\textup{(}$\mathfrak{a}\in I(\mathcal{O},\,\ell_\mathcal{O}N)$\textup{)}. 
\end{corollary} 
\begin{proof}
The result follows from Propositions \ref{lNK} and \ref{lNN}. 
\end{proof}

We shall denote by $K_{\mathcal{O},\,G}$ the unique abelian extension $K$ in which
every ramified prime of $K$ divides $\ell_\mathcal{O}N\mathcal{O}_K$ and
$I(\mathcal{O}_K,\,\ell_\mathcal{O}N)/P_G(\mathcal{O}_K,\,\ell_\mathcal{O}N,\,
\ell_\mathcal{O}N)\simeq\mathrm{Gal}(K_{\mathcal{O},\,G}/K)$ via the Artin map 
for the modulus $\ell_\mathcal{O}N\mathcal{O}_K$. 
Then we deduce by Corollary \ref{CK} that
\begin{equation*}
\mathrm{Gal}(K_{\mathcal{O},\,G}/K)\simeq\mathcal{C}_G(\mathcal{O},\,N). 
\end{equation*}

\section {Generation of  class fields}

By utilizing Shimura's theory of canonical models, 
we shall construct the field $K_{\mathcal{O},\,G}$ over $K$.
\par
Observe that $K$ has no real embedding. 
Let 
\begin{equation*}
\widehat{\mathbb{Z}}=\prod_{p\,:\,\textrm{primes}}\mathbb{Z}_p,\quad
\widehat{K}=K\otimes_\mathbb{Z}\widehat{\mathbb{Z}}
\quad\textrm{and}\quad
\widehat{\mathcal{O}}_K=\mathcal{O}_K\otimes_\mathbb{Z}\widehat{\mathbb{Z}}.
\end{equation*}
The group of (finite) $K$-ideles is defined by the group of units $\widehat{K}^\times$ in $\widehat{K}$. 
For an idele $s\in\widehat{K}^\times$, we
mean by $s\mathcal{O}_K$ the fractional ideal $K\cap s\widehat{\mathcal{O}}_K$
of $K$.
Through the natural inclusion 
$\widehat{K}\hookrightarrow\prod_p(K\otimes_\mathbb{Z}\mathbb{Z}_p)$,
we see that
\begin{equation}\label{identify}
\widehat{K}^\times\simeq\left\{
s=(s_p)_p\in\prod_{p\,:\,\textrm{primes}}(K\otimes_\mathbb{Z}\mathbb{Z}_p)^\times~|~
s_p\in(\mathcal{O}_K\otimes_\mathbb{Z}\mathbb{Z}_p)^\times~
\textrm{for all but finitely many $p$}
\right\}.
\end{equation}
Hence we shall identify $\widehat{K}^\times$ with the above subgroup of 
$\prod_p(K\otimes_\mathbb{Z}\mathbb{Z}_p)^\times$. 
Let $K^\mathrm{ab}$ be the maximal abelian extension of $K$. 

\begin{proposition}\label{CFT}
The Artin  map $\widehat{K}^\times\rightarrow\mathrm{Gal}(K^\mathrm{ab}/K)$ yields a one-to-one correspondence
\begin{eqnarray*}
\{\textrm{closed subgroups of $\widehat{K}^\times$ of finite index containing $K^\times$}\} 
& \rightarrow & \{\textrm{finite abelian extensions of $K$}\}\\
J & \mapsto &  \textrm{$L$ satisfying $\widehat{K}^\times/J\simeq\mathrm{Gal}(L/K)$}.
\end{eqnarray*}
\end{proposition}
\begin{proof}
See \cite[$\S$IV.7]{Neukirch}. 
\end{proof}

Now, we set
\begin{eqnarray*}
\mathbb{Z}_G&=&\{t\in\mathbb{Z}~|~0\leq t<N~\textrm{and}~t+N\mathbb{Z}\in G\},\\
\Gamma_G&=&\left\{\gamma\in\mathrm{SL}_2(\mathbb{Z})~|~\gamma\equiv\begin{bmatrix}
t& \mathrm{*}\\0 & \mathrm{*}\end{bmatrix}\Mod{NM_2(\mathbb{Z})}~
\textrm{for some}~t\in\mathbb{Z}_G
\right\},\\
\mathcal{F}_{\Gamma_G,\,\mathbb{Q}}&=&\textrm{the field of meromorphic modular functions for $\Gamma_G$
with rational Fourier coefficients}.
\end{eqnarray*}
For each prime $p$, we let
$\mathcal{O}_{K,\,p}=\mathcal{O}_K\otimes_\mathbb{Z}\mathbb{Z}_p$
and
 $\mathcal{O}_p=\mathcal{O}\otimes_\mathbb{Z}\mathbb{Z}_p$. 

\begin{lemma}\label{idelegroup}
In the sense of \textup{Proposition \ref{CFT}}, we establish the following two correspondences. 
\begin{enumerate}
\item[\textup{(i)}] The field $K_{\mathcal{O},\,\{1+N\mathbb{Z}\}}$ corresponds to the
subgroup 
\begin{equation*}
J_{\mathcal{O},\,\{1+N\mathbb{Z}\}}=K^\times\left\{
\prod_{p\,|\,N}(1+N\mathcal{O}_p)\times\prod_{p\,\nmid\,N}\mathcal{O}_p^\times\right\}~\textrm{of}~
\widehat{K}^\times.
\end{equation*}
\item[\textup{(ii)}]
The field $K_{\mathcal{O},\,G}$ corresponds to the subgroup
\begin{equation*}
J_{\mathcal{O},\,G}=\bigcup_{t\in\mathbb{Z}_G}K^\times\left\{
\prod_{p\,|\,N}(t+N\mathcal{O}_p)\times
\prod_{p\,\nmid\,N}\mathcal{O}_p^\times\right\}~\textrm{of}~
\widehat{K}^\times.
\end{equation*}
\end{enumerate}
\end{lemma}
\begin{proof}
\begin{enumerate}
\item[(i)] See \cite[Lemma 15.20]{Cox}.
\item[(ii)] 
For simplicity, let 
$P_G=P_G(\mathcal{O}_K,\,\ell_\mathcal{O}N,\,\ell_\mathcal{O}N)$. 
The Artin map for the modulus $\ell_\mathcal{O}N\mathcal{O}_K$ induces
the isomorphism
\begin{equation*}
I(\mathcal{O}_K,\,\ell_\mathcal{O}N)/P_{\{1+N\mathbb{Z}\}}
\stackrel{\sim}{\rightarrow}\mathrm{Gal}(K_{\mathcal{O},\,G}/K). 
\end{equation*}
Therefore, by Proposition \ref{CFT} and (i) the Artin map $\widehat{K}^\times\rightarrow\mathrm{Gal}(K^\mathrm{ab}/K)$ yields the surjection
\begin{eqnarray*}
\phi~:~\widehat{K}^\times&\rightarrow& I(\mathcal{O}_K,\,\ell_\mathcal{O}N)/P_{\{1+N\mathbb{Z}\}}\\
s & \mapsto & [\nu_ss\mathcal{O}_K] 
\end{eqnarray*}
with $\mathrm{ker}(\phi)=J_{\mathcal{O},\,\{1+N\mathbb{Z}\}}$. 
Here, $\nu_s$ is any element of $K^\times$ so that 
\begin{equation*}
\nu_ss_p\in 1+\ell_\mathcal{O}N\mathcal{O}_{K,\,p}\quad\textrm{for all}~p\,|\,\ell_\mathcal{O}N
\end{equation*} 
which can be taken by the approximation theorem (\cite[Chapter IV]{Janusz}). 
For each $t\in\mathbb{Z}_G$, 
choose a pair of nonzero integers $t_1$ and $t_2$ such that
\begin{equation*}
\gcd(t_1,\,\ell_\mathcal{O}N)=
\gcd(t_2,\,\ell_\mathcal{O}N)=1,~
t_1\equiv t\Mod{N}~\textrm{and}~t_1t_2\equiv1\Mod{\ell_\mathcal{O}N} 
\end{equation*}
by using the fact that the reduction $(\mathbb{Z}/\ell_\mathcal{O}N\mathbb{Z})^\times\rightarrow(\mathbb{Z}/N\mathbb{Z})^\times$ is surjective.
Let $s=s(t)=(s_p)_p$ be the element of $\widehat{K}^\times$ defined by
\begin{equation}\label{s_p}
s_p=\left\{\begin{array}{ll}
t_1 & \textrm{if}~p\,|\,\ell_\mathcal{O}N,\\
1 & \textrm{if}~p\nmid\ell_\mathcal{O}N.
\end{array}\right.
\end{equation}
Furthermore, if we let $\nu_s=t_2$ ($\in K^\times$), then we see that
\begin{equation*}
\nu_s s_p\in 1+\ell_\mathcal{O}N\mathcal{O}_{K,\,p}\quad\textrm{for all}~p\,|\,\ell_\mathcal{O}N,
\end{equation*}
and hence in the generalized ideal class group 
$I(\mathcal{O}_K,\,\ell_\mathcal{O}N)/P_{\{1+N\mathbb{Z}\}}$
\begin{equation}\label{tst}
\phi(s)=[\nu_ss\mathcal{O}_K]=[t_2s\mathcal{O}_K]=[t_2\mathcal{O}_K]. 
\end{equation}
We deduce from the inclusion $P_G\supseteq P_{\{1+N\mathbb{Z}\}}$ that
\begin{equation*}
K_{\mathcal{O},\,G}\subseteq K_{\mathcal{O},\,\{1+N\mathbb{Z}\}}\quad\textrm{and}\quad 
\mathrm{Gal}(K_{\mathcal{O},\,\{1+N\mathbb{Z}\}}/K_{\mathcal{O},\,G})\simeq
P_G/P_{\{1+N\mathbb{Z}\}}. 
\end{equation*}
We then find that
\begin{eqnarray*}
\phi^{-1}(P_G/P_{\{1+N\mathbb{Z}\}})&=&
\phi^{-1}(\{
[t_1\mathcal{O}_K]\,|\,t\in\mathbb{Z}_G \})\quad\textrm{by the definitions of $P_G$ and
$P_{\{1+N\mathbb{Z}\}}$}\\
&=&
\phi^{-1}(\{
[t_2\mathcal{O}_K]\,|\,t\in\mathbb{Z}_G \})\quad\textrm{because $G$ is a 
subgroup of $(\mathbb{Z}/N\mathbb{Z})^\times$}\\
&=&\bigcup_{t\in\mathbb{Z}_G}s(t)\mathrm{ker}(\phi)\quad\textrm{by (\ref{tst})}\\
&=&\bigcup_{t\in\mathbb{Z}_G}K^\times
\left\{
\prod_{p\,|\,N}(t_1+N\mathcal{O}_p)\times
\prod_{p\,\nmid\,N,~p\,|\,\ell_\mathcal{O}}t_1\mathcal{O}_p^\times\times
\prod_{p\,\nmid\,\ell_\mathcal{O}N}\mathcal{O}_p^\times\right\}\\
&&\hspace{4cm}\textrm{by  the fact $\mathrm{ker}(\phi)=J_{\mathcal{O},\,\{1+N\mathbb{Z}\}}$ and (\ref{s_p})}\\
&=&\bigcup_{t\in\mathbb{Z}_G}K^\times
\left\{\prod_{p\,|\,N}(t+N\mathcal{O}_p)\times
\prod_{p\,\nmid\,N}\mathcal{O}_p^\times\right\}.
\end{eqnarray*}
\end{enumerate}
\end{proof}

Let $\mathcal{F}_N$ be the field of meromorphic modular functions for the principal congruence subgroup 
$\Gamma(N)=\{\alpha\in\mathrm{SL}_2(\mathbb{Z})~|~\alpha\equiv I_2\Mod{NM_2(\mathbb{Z})}\}$
whose Fourier coefficients belong to $\mathbb{Q}(\zeta_N)$. 
As is well known, $\mathcal{F}_N$ is Galois over $\mathcal{F}_1$ and $\mathrm{Gal}(\mathcal{F}_N/\mathcal{F}_1)
\simeq\mathrm{GL}_2(\mathbb{Z}/N\mathbb{Z})/\langle-I_2\rangle$. More precisely, let $\gamma\in\mathrm{GL}_2(\mathbb{Z}/N\mathbb{Z})/\langle-I_2\rangle$ \textup{(}$\simeq\mathrm{Gal}(\mathcal{F}_N/\mathcal{F}_1)$\textup{)}
and $h\in\mathcal{F}_N$ with Fourier expansion 
\begin{equation*}
h(\tau)=\sum_{n\gg-\infty}c_nq^{n/N}\quad(c_n\in\mathbb{Q}(\zeta_N),\,\tau\in\mathbb{H},\,q=e^{2\pi\mathrm{i}\tau}). 
\end{equation*}
 If $\gamma\in\mathrm{SL}_2(\mathbb{Z}/N\mathbb{Z})/\langle-I_2\rangle$, then
$h^\gamma=h\circ\widetilde{\gamma}$ where $\widetilde{\gamma}$ is any element of $\mathrm{SL}_2(\mathbb{Z})$
which reduces to $\gamma$. 
On the other hand, if $\gamma$ is obtained by reducing
$\begin{bmatrix}1&0\\0&d\end{bmatrix}$ for some integer $d$ relatively prime to $N$, then
$h^\gamma=\sum_nc_n^{\sigma_d}q^{n/N}$ where
$\sigma_d$ is the automorphism of $\mathbb{Q}(\zeta_N)$ defined by $\zeta_N\mapsto\zeta_N^d$ 
 (\cite[Theorem 3 in Chapter 6]{Lang87}
and \cite[Proposition 6.9 (1)]{Shimura}).
Now, if we set 
\begin{equation*}
\mathcal{F}=\bigcup_{N=1}^\infty\mathcal{F}_N\quad\textrm{and}\quad
\widehat{\mathbb{Q}}=\mathbb{Q}\otimes_\mathbb{Z}\widehat{\mathbb{Z}}, 
\end{equation*}
then we get an exact sequence
\begin{equation*}
1\rightarrow\mathbb{Q}^\times\rightarrow\mathrm{GL}_2(\widehat{\mathbb{Q}})
\rightarrow\mathrm{Gal}(\mathcal{F}/\mathbb{Q})\rightarrow 1
\end{equation*}
(\cite[Theorem 2 in Chapter 7 and p. 79]{Lang87} or \cite[Theorem 6.23]{Shimura}).

Let $\omega\in K\cap\mathbb{H}$. We define an embedding
\begin{equation*}
q_\omega:K^\times\rightarrow\mathrm{GL}_2^+(\mathbb{Q})
\end{equation*}
by using the relation 
\begin{equation}\label{qwrelation}
\tau\begin{bmatrix}\omega\\1\end{bmatrix}=
q_\omega(\tau)\begin{bmatrix}\omega\\1\end{bmatrix}\quad(\tau\in K^\times).
\end{equation}
By continuity $q_\omega$ can be extended to an embedding
$(K\otimes_\mathbb{Z}\mathbb{Z}_p)^\times\rightarrow\mathrm{GL}_2(\mathbb{Q}_p)$
for each prime $p$. Thus we obtain an embedding
\begin{equation*}
q_\omega:\widehat{K}^\times\rightarrow\mathrm{GL}_2(\widehat{\mathbb{Q}}). 
\end{equation*}
For an open subgroup $S$ of $\mathrm{GL}_2(\widehat{\mathbb{Q}})$ containing scalars $\mathbb{Q}^\times$
such that $S/\mathbb{Q}^\times$ is compact, we define
\begin{eqnarray*}
\Gamma_S&=&S\cap\mathrm{GL}_2^+(\mathbb{Q}),\\
\mathcal{F}_S&=&\{h\in\mathcal{F}~|~h^\gamma=h~\textrm{for all}~\gamma\in S\},\\
k_S&=&\{\nu\in\mathbb{Q}^\mathrm{ab}~|~\nu^{[s,\,\mathbb{Q}]}=\nu~\textrm{for all}~s\in\mathbb{Q}^\times
\det(S)\,(\subset\widehat{\mathbb{Q}}^\times)\}.
\end{eqnarray*}
Here, $\mathbb{Q}^\mathrm{ab}$ is
the maximal abelian extension of $\mathbb{Q}$
and $[\,\cdot\,,\,\mathbb{Q}]$ is the Artin map for $\mathbb{Q}$. 
The theory of canonical models yields the following proposition. 

\begin{proposition}\label{canonical}
With the above notations, we have
\begin{enumerate}
\item[\textup{(i)}] $\Gamma_S/\mathbb{Q}^\times$ is a Fuchsian group of the first kind
commensurable with $\mathrm{SL}_2(\mathbb{Z})/\langle-I_2\rangle$. 
\item[\textup{(ii)}] $\mathbb{C}\mathcal{F}_S$ is the field of meromorphic modular functions
for $\Gamma_S/\mathbb{Q}^\times$. 
\item[\textup{(iii)}] $k_S$ is algebraically closed in $\mathcal{F}_S$. 
\item[\textup{(iv)}] If $\omega\in K\cap\mathbb{H}$, then the subgroup $K^\times q_\omega^{-1}(S)$
of $\widehat{K}^\times$ corresponds to the subfield
\begin{equation*}
K(h(\omega)~|~h\in\mathcal{F}_S~\textrm{is finite at}~\omega)
\end{equation*}
of $K^\mathrm{ab}$ in the sense of \textup{Proposition \ref{CFT}}.
\end{enumerate}
\end{proposition}
\begin{proof}
See \cite[Propositions 6.27 and 6.33]{Shimura}. 
\end{proof}

\begin{lemma}\label{kSQ}
With the notations as in \textup{Proposition \ref{canonical}},  we get the following.  
\begin{enumerate}
\item[\textup{(i)}] 
If $S$ contains $\left\{
\begin{bmatrix}1&0\\0&d\end{bmatrix}~|~d\in\widehat{\mathbb{Z}}^\times\right\}$, then $\mathcal{F}_S$ coincides with
the field of meromorphic modular functions for $\Gamma_S/\mathbb{Q}^\times$ with rational Fourier coefficients. 
\item[\textup{(ii)}] Let $\omega$ and $\omega'$ be elements of $K\cap\mathbb{H}$ such that
if $h\in\mathcal{F}_S$ is finite at $\omega$, then it is also finite at $\omega'$. If
$h(\omega)=h(\omega')$ for all $h\in\mathcal{F}_S$ finite at $\omega$, 
then $\omega=\gamma(\omega')$ for some $\gamma\in\Gamma_S/\mathbb{Q}^\times$.
Here, $\gamma$ acts on $\omega'$ as a fractional linear transformation. 
\end{enumerate}
\end{lemma}
\begin{proof}
\begin{enumerate}
\item[(i)] See \cite[Theorem 26.4]{Shimura98}. 
\item[(ii)] See \cite[Lemma 7.1]{J-K-S-Y223} and \cite[Theorem 26.4]{Shimura98}. 
\end{enumerate}
\end{proof}

\begin{theorem}\label{generation}
We have $K_{\mathcal{O},\,G}=K\left(
h(\tau_\mathcal{O})~|~h\in\mathcal{F}_{\Gamma_G,\,\mathbb{Q}}~\textrm{is finite at}~\tau_\mathcal{O}\right)$. 
\end{theorem}
\begin{proof}
Set $S=\mathbb{Q}^\times W$ with
\begin{equation*}
W=\left\{
(\gamma_p)_p\in\prod_p\mathrm{GL}_2(\mathbb{Z}_p)~\Bigg|~
\begin{array}{l}\textrm{there is an element $t$ of $\mathbb{Z}_G$ such that}\\
\gamma_p\equiv
\begin{bmatrix}
t & \mathrm{*}\\
0& \mathrm{*}\end{bmatrix}\Mod{NM_2(\mathbb{Z}_p)}~\textrm{for all primes $p$}
\end{array}\right\}.
\end{equation*}
Here, one can readily verify that
$\Gamma_S=\mathbb{Q}^\times\Gamma_G$ and
$S\supseteq\left\{\begin{bmatrix}1&0\\0&d\end{bmatrix}~|~d\in\widehat{\mathbb{Z}}^\times\right\}$.
So we get that
$\Gamma_S/\mathbb{Q}^\times\simeq\langle\Gamma_G,\,-I_2\rangle/\langle-I_2\rangle$
and $k_S=\mathbb{Q}$, and hence
\begin{equation}\label{FF}
\mathcal{F}_S=\mathcal{F}_{\Gamma_G,\,\mathbb{Q}}
\end{equation}
by Proposition \ref{canonical} (ii) and Lemma \ref{kSQ} (i). 
\par
Let $\mathrm{irr}(\tau_\mathcal{O},\,\mathbb{Q})=x^2+b_\mathcal{O}x+c_\mathcal{O}$
($\in\mathbb{Z}[x]$), and so
\begin{equation*}
\tau_\mathcal{O}^2+b_\mathcal{O}\tau_\mathcal{O}+c_\mathcal{O}=0,\quad
\tau_\mathcal{O}+\overline{\tau}_\mathcal{O}=-b_\mathcal{O}\quad\textrm{and}\quad
\tau_\mathcal{O}\overline{\tau}_\mathcal{O}=c_\mathcal{O}.
\end{equation*}
Since $\mathcal{O}_p=\mathcal{O}\otimes_\mathbb{Z}\mathbb{Z}_p=
\mathbb{Z}_p\tau_\mathcal{O}+\mathbb{Z}_p$, we derive that
if $s=(s_p)_p$ is an element of $\widehat{K}^\times$ 
such that $s_p\in\mathcal{O}_p$ for every prime $p$, then 
$s_p=u_p\tau_\mathcal{O}+v_p$ for some $u_p,\,v_p\in\mathbb{Z}_p$ and
\begin{equation}\label{q(s)}
q_{\tau_\mathcal{O}}(s)=(\gamma_p)_p\quad\textrm{with}~
\gamma_p=\begin{bmatrix}v_p-b_\mathcal{O}u_p & -c_\mathcal{O}u_p\\
u_p & v_p\end{bmatrix}.
\end{equation}
We further observe that
\begin{equation}\label{inGL_2(Z_p)}
s_p\in\mathcal{O}_p^\times~\Longleftrightarrow~
\gamma_p\in\mathrm{GL}_2(\mathbb{Z}_p)~\Longleftrightarrow~
\det(\gamma_p)=(u_p\tau_\mathcal{O}+v_p)(u_p\overline{\tau}_\mathcal{O}+v_p)\in\mathbb{Z}_p^\times.
\end{equation}
Then we find that
\begin{eqnarray*}
&&K^\times q_{\tau_\mathcal{O}}^{-1}(S)\\&=&K^\times\left\{
s=(s_p)_p\in\widehat{K}^\times~|~q_{\tau_\mathcal{O}}(s)\in S=\mathbb{Q}^\times W\right\}\\
&=&K^\times\left\{
s=(s_p)_p\in\widehat{K}^\times~|~s_p\in\mathcal{O}_p~\textrm{and}~q_{\tau_\mathcal{O}}(s)\in W\right\}\\
&&\hspace{4cm}\textrm{by (\ref{identify}) and the fact}~q_{\tau_\mathcal{O}}(r)=rI_2~\textrm{for every}~r\in\mathbb{Q}^\times~
\textrm{from (\ref{qwrelation})}\\
&=&K^\times\left\{
s=(s_p)_p\in\widehat{K}^\times~\Bigg|~
\begin{array}{l}
s_p=u_p\tau_\mathcal{O}+v_p~\textrm{with}~u_p,\,v_p\in\mathbb{Z}_p~\textrm{such that}\\
\gamma_p=\left[\begin{smallmatrix}v_p-b_\mathcal{O}u_p & -c_\mathcal{O}u_p\\
u_p & v_p\end{smallmatrix}\right]\in W\end{array}
\right\}\quad\textrm{by (\ref{q(s)})}\\
&=&\bigcup_{t\in\mathbb{Z}_G}K^\times\left\{
s=(s_p)_p\in\widehat{K}^\times~\Bigg|~\begin{array}{l}s_p=u_p\tau_\mathcal{O}+v_p~\textrm{with}~u_p,\,v_p\in\mathbb{Z}_p~
\textrm{such that}\\
\gamma_p=\left[\begin{smallmatrix}v_p-b_\mathcal{O}u_p & -c_\mathcal{O}u_p\\
u_p & v_p\end{smallmatrix}\right]\in\mathrm{GL}_2(\mathbb{Z}_p)~\textrm{and}~
\gamma_p\equiv
\left[\begin{smallmatrix}t&\mathrm{*}\\0&\mathrm{*}\end{smallmatrix}\right]\Mod{NM_2(\mathbb{Z}_p)}
\end{array}
\right\}\\
&=&\bigcup_{t\in\mathbb{Z}_G}K^\times\left\{
s=(s_p)_p\in\widehat{K}^\times~\Bigg|~\begin{array}{l}s_p=u_p\tau_\mathcal{O}+v_p~\textrm{with}~u_p,\,v_p\in\mathbb{Z}_p~
\textrm{such that}\\
s_p\in\mathcal{O}_p^\times,\,
u_p\equiv0\Mod{N\mathbb{Z}_p}~\textrm{and}~v_p\equiv t\Mod{N\mathbb{Z}_p}
\end{array}
\right\}
\quad\textrm{by (\ref{inGL_2(Z_p)})}\\
&=&\bigcup_{t\in\mathbb{Z}_G}K^\times\left\{
\prod_{p\,|\,N}(t+N\mathcal{O}_p)\times
\prod_{p\,\nmid\,N}\mathcal{O}_p^\times\right\}.
\end{eqnarray*}
Therefore, we conclude by Lemma \ref{idelegroup} (ii), Proposition \ref{canonical} (iv) and (\ref{FF}) and  that
\begin{equation*}
K_{\mathcal{O},\,G}=K\left(
h(\tau_\mathcal{O})~|~h\in\mathcal{F}_{\Gamma_G,\,\mathbb{Q}}~\textrm{is finite at}~\tau_\mathcal{O}\right).
\end{equation*}
\end{proof}

\begin{remark}
Theorem \ref{generation} for the case $G=\{1+N\mathbb{Z}\}$
was given in \cite[Theorems 4 and 5]{Cho}. 
\end{remark}

\section {Form class groups of level $N$}

We shall introduce the notion of form class groups of level $N$ 
in view of $\mathcal{O}$-ideal class groups.
\par
The modular group $\mathrm{SL}_2(\mathbb{Z})$ acts on the set
\begin{equation*}
\mathcal{Q}(D_\mathcal{O})=\{ax^2+bxy+cy^2\in\mathbb{Z}[x,\,y]~|~
\gcd(a,\,b,\,c)=1,~b^2-4ac=D_\mathcal{O},~a>0\}
\end{equation*}
from the right by
\begin{equation*}
Q^\gamma=Q\left(\begin{bmatrix}x\\y\end{bmatrix}\right)^\gamma=
Q\left(\gamma\begin{bmatrix}x\\y\end{bmatrix}\right)\quad(Q\in\mathcal{Q}(D_\mathcal{O}),~\gamma\in\mathrm{SL}_2(\mathbb{Z})). 
\end{equation*}
Let $\Gamma$ be a congruence subgroup of $\mathrm{SL}_2(\mathbb{Z})$ of level $N$, that is, 
$\Gamma $ is a subgroup of $\mathrm{SL}_2(\mathbb{Z})$ containing $\Gamma(N)$.
We then obtain an equivalence relation $\sim_\Gamma$ on the set 
\begin{equation*}
\mathcal{Q}(D_\mathcal{O},\,N)=\{ax^2+bxy+cy^2\in\mathcal{Q}(D_\mathcal{O})~|~
\gcd(a,\,N)=1\}
\end{equation*}
as follows\,:
for $Q,\,Q'\in\mathcal{Q}(D_\mathcal{O},\,N)$ 
\begin{equation*}
Q\sim_\Gamma Q'\quad\Longleftrightarrow\quad Q'=Q^\gamma~\textrm{for some}~\gamma\in\Gamma. 
\end{equation*}
Here, $Q^\gamma$ means the action of $\gamma\in\mathrm{SL}_2(\mathbb{Z})$ on the binary quadratic form $Q\in\mathcal{Q}(D_\mathcal{O})$. 
Denote the set of equivalence classes by $\mathcal{C}_\Gamma(D_\mathcal{O},\,N)$, namely,
\begin{equation*}
\mathcal{C}_\Gamma(D_\mathcal{O},\,N)=\mathcal{Q}(D_\mathcal{O},\,N)/\sim_\Gamma.
\end{equation*}
For $Q=ax^2+bxy+cy^2\in\mathcal{Q}(D_\mathcal{O},\,N)$, let $\omega_Q$ be the element of $\mathbb{H}$ defined in (\ref{w_Q}). 
In particular, if we let $Q_0=x^2+b_\mathcal{O}xy+c_\mathcal{O}y^2$ be the principal form in $\mathcal{Q}(D_\mathcal{O},\,N)$,
then we see that
\begin{equation}\label{wtO}
\omega_{Q_0}=\tau_\mathcal{O}\quad\textrm{and}\quad
[\omega_{Q_0},\,1]=\mathcal{O}. 
\end{equation}

\begin{lemma}\label{ION}
The lattice $[\omega_Q,\,1]=\mathbb{Z}\omega_Q+\mathbb{Z}$ in $\mathbb{C}$ belongs to $I(\mathcal{O},\,N)$. 
\end{lemma}
\begin{proof}
See \cite[Lemma 9.1]{J-K-S-Y23}. 
\end{proof}

Let  $\Gamma_1(N)$ be the congruence subgroup of $\mathrm{SL}_2(\mathbb{Z})$ defined by
\begin{equation*}
\Gamma_1(N)=\left\{\gamma\in\mathrm{SL}_2(\mathbb{Z})~|~\gamma\equiv\begin{bmatrix}
1&\mathrm{*}\\0&1\end{bmatrix}\Mod{NM_2(\mathbb{Z})}\right\}. 
\end{equation*}
Note that
$\Gamma_{\{1+N\mathbb{Z}\}}=\Gamma_1(N)$ and 
 $\Gamma_1(N)=\langle\Gamma(N),\,T\rangle$ ($\leq\mathrm{SL}_2(\mathbb{Z})$), where 
$T=\begin{bmatrix}1&1\\0&1\end{bmatrix}$. 
Furthermore, we observe that if $Q\in\mathcal{Q}(D_\mathcal{O},\,N)$, then
\begin{equation*}
[\omega_{Q^T},\,1]=[T^{-1}(\omega_Q),\,1]=[\omega_Q-1,\,1]=[\omega_Q,\,1]. 
\end{equation*}

\begin{proposition}\label{Gamma_1(N)}
One can give the set $\mathcal{C}_{\Gamma_1(N)}(D_\mathcal{O},\,N)$
a unique group structure so that
the mapping
\begin{eqnarray*}
\psi_{\Gamma_1(N),\,P_{\{1+N\mathbb{Z}\}}(\mathcal{O},\,N)}~:~\mathcal{C}_{\Gamma_1(N)}(D_\mathcal{O},\,N)&\rightarrow&
I(\mathcal{O},\,N)/P_{\{1+N\mathbb{Z}\}}(\mathcal{O},\,N)\\
\mathrm{[}Q] &\mapsto& [[\omega_Q,\,1]]
\end{eqnarray*}
becomes a well-defined isomorphism. 
\end{proposition}
\begin{proof}
See \cite[Definition 5.7 and Proposition 9.3]{J-K-S-Y23}. 
\end{proof}

\begin{definition}\label{induce}
We say that
\textit{$\Gamma$ induces a form class group of level $N$} if
\begin{enumerate}
\item[(i)] $\Gamma$ contains $\Gamma_1(N)$, 
\item[(ii)] there is a subgroup $P$ of $P(\mathcal{O},\,N)$ containing $P_{\{1+N\mathbb{Z}\}}(\mathcal{O},\,N)$
in order for the mapping
\begin{eqnarray*}
\psi_{\Gamma,\,P}~:~\mathcal{C}_\Gamma(D_\mathcal{O},\,N) &\rightarrow& I(\mathcal{O},\,N)/P\\
\mathrm{[}Q] &\mapsto& [[\omega_Q,\,1]] 
\end{eqnarray*}
to be a well-defined bijection.
\end{enumerate}
In this case, we regard the set $\mathcal{C}_\Gamma(D_\mathcal{O},\,N)$ as a group isomorphic to the quotient group 
$I(\mathcal{O},\,N)/P$, and call it a \textit{form class group of level $N$}.  
\end{definition}

\begin{remark}\label{Q0}
Suppose that $\Gamma$ induces a form class group of level $N$. 
\begin{enumerate}
\item[(i)]
Note by (\ref{wtO}) that the identity element of the form class group $\mathcal{C}_\Gamma(D_\mathcal{O},\,N)$ is $[Q_0]$. 
\item[(ii)] The natural surjection $\mathcal{C}_{\Gamma_1(N)}(D_\mathcal{O},\,N)
\rightarrow\mathcal{C}_\Gamma(D_\mathcal{O},\,N)$ is indeed a group homomorphism
by the commutative diagram in Figure \ref{Figure02}.
\begin{figure}[h]
\begin{equation*}
\xymatrixcolsep{12pc}
\xymatrix{
\mathcal{C}_{\Gamma_1(N)}(D_\mathcal{O},\,N) \ar@{->}[r]_{\psi_{\Gamma_1(N),\,P_{\{1+N\mathbb{Z}\}}(\mathcal{O},\,N)}}^{\sim}
\ar@{->>}[dd]_{\textrm{natural surjection}}
 & I(\mathcal{O},\,N)/P_{\{1+N\mathbb{Z}\}}(\mathcal{O},\,N) \ar@{->>}[dd]^{\textrm{natural homomorphism}} \\\\
\mathcal{C}_\Gamma(D_\mathcal{O},\,N) \ar@{->}[r]^{\sim}_{\psi_{\Gamma,\,P}} & I(\mathcal{O},\,N)/P
}
\end{equation*}
\caption{A commutative diagram showing that the natural surjection is a homomorphism}
\label{Figure02}
\end{figure}
\end{enumerate}
\end{remark}

For $\gamma=\begin{bmatrix}a&b\\c&d\end{bmatrix}\in\mathrm{SL}_2(\mathbb{Z})$
and $\tau\in\mathbb{H}$, we denote by
\begin{equation*}
j(\gamma,\,\tau)=c\tau+d.
\end{equation*}
Define the subgroup $P_\Gamma$ of $P(\mathcal{O})$ by
\begin{equation*}
P_\Gamma=\langle j(\gamma^{-1},\,\omega_Q)\mathcal{O}~|~Q\in\mathcal{Q}(D_\mathcal{O},\,N)~
\textrm{and}~\gamma\in\Gamma~\textrm{such that}~Q^\gamma\in\mathcal{Q}(D_\mathcal{O},\,N)\rangle.
\end{equation*}
Since
\begin{equation*}
[\omega_{Q^\gamma},\,1]=[\gamma^{-1}(\omega_Q),\,1]
=\frac{1}{j(\gamma^{-1},\,\omega_Q)}[\omega_Q,\,1],
\end{equation*}
$P_\Gamma$ is a subgroup of $P(\mathcal{O},\,N)$ by Lemma \ref{ION}. 

\begin{lemma}\label{well-defined}
When $P$ is a subgroup of $P(\mathcal{O},\,N)$, we see that the mapping 
\begin{eqnarray*}
\psi_{\Gamma,\,P}~:~\mathcal{C}_\Gamma(D_\mathcal{O},\,N) & \rightarrow & I(\mathcal{O},\,N)/P\\
\mathrm{[}Q] & \mapsto & [[\omega_Q,\,1]]
\end{eqnarray*}
is well defined if and only if $P$ contains $P_\Gamma$. 
\end{lemma}
\begin{proof}
Assume that $\psi_{\Gamma,\,P}$ is well defined. 
Let $Q\in\mathcal{Q}(D_\mathcal{O},\,N)$ and $\gamma\in\Gamma$ such that $Q^\gamma\in\mathcal{Q}(D_\mathcal{O},\,N)$. Since
$[Q]=[Q^\gamma]$ in $\mathcal{C}_\Gamma(D_\mathcal{O},\,N)$, we claim that in the quotient group $I(\mathcal{O},\,N)/P$
\begin{equation*}
[[\omega_Q,\,1]]=[[\omega_{Q^\gamma},\,1]]=[[\gamma^{-1}(\omega_Q),\,1]]
=\left[\frac{1}{j(\gamma^{-1},\,\omega_Q)}[\omega_Q,\,1]\right].
\end{equation*}
This implies that $j(\gamma^{-1},\,\omega_Q)\mathcal{O}\in P$, and hence $P$ contains $P_\Gamma$. 
\par
Conversely, assume that $P$ contains $P_\Gamma$. Let $Q,\,Q'\in\mathcal{Q}(D_\mathcal{O},\,N)$ such that
$[Q]=[Q']$ in $\mathcal{C}_\Gamma(D_\mathcal{O},\,N)$, and so $Q'=Q^\gamma$ for some $\gamma\in\Gamma$. 
Since $j(\gamma^{-1},\,\omega_Q)\mathcal{O}\in P_\Gamma\subseteq P$, we obtain that in 
the quotient group
$I(\mathcal{O},\,N)/P$
\begin{equation*}
[[\omega_{Q'},\,1]]=[[\omega_{Q^\gamma},\,1]]=[[\gamma^{-1}(\omega_Q),\,1]]
=\left[\frac{1}{j(\gamma^{-1},\,\omega_Q)}[\omega_Q,\,1]\right]=[[\omega_Q,\,1]].
\end{equation*}
Thus $\psi_{\Gamma,\,P}$ is well defined. 
\end{proof}

For a subgroup $P$ of $P(\mathcal{O},\,N)$ containing $P_\Gamma$, 
let $\psi_{\Gamma,\,P}$ be the well-defined map stated in Lemma \ref{well-defined}. 
Let
$\widetilde{P}_\Gamma$ be the subgroup of $P(\mathcal{O},\,N)$ given by
\begin{equation*}
\widetilde{P}_\Gamma=P_\Gamma P_{\{1+N\mathbb{Z}\}}(\mathcal{O},\,N).
\end{equation*}

\begin{proposition}\label{theoretical}
If $\Gamma$ is a subgroup of $\mathrm{SL}_2(\mathbb{Z})$ containing $\Gamma_1(N)$, then 
\begin{equation*}
\textrm{$\Gamma$ induces a form class group of level $N$}\quad\Longleftrightarrow\quad
\textrm{$\psi_{\Gamma,\,\widetilde{P}_\Gamma}$ is injective}.
\end{equation*}
In this case, $\widetilde{P}_\Gamma$ is a unique subgroup $P$ of $P(\mathcal{O},\,N)$ satisfying the condition 
\textup{(ii)} of \textup{Definition \ref{induce}}. 
\end{proposition}
\begin{proof}
Note that
$\psi_{\Gamma,\,\widetilde{P}_\Gamma}$ is surjective
by the commutative diagram in Figure \ref{Figure03} which is derived from Proposition \ref{Gamma_1(N)} and Lemma \ref{well-defined}. 
\begin{figure}[h]
\begin{equation*}
\xymatrixcolsep{12pc}
\xymatrix{
\mathcal{C}_{\Gamma_1(N)}(D_\mathcal{O},\,N) \ar@{->}[r]_{\psi_{\Gamma_1(N),\,P_{\{1+N\mathbb{Z}\}}(\mathcal{O},\,N)}}^{\sim}
\ar@{->>}[dd]_{\textrm{natural}}
 & I(\mathcal{O},\,N)/P_{\{1+N\mathbb{Z}\}}(\mathcal{O},\,N) \ar@{->>}[dd]^{\textrm{natural}} \\\\
\mathcal{C}_\Gamma(D_\mathcal{O},\,N) \ar@{->}[r]_{\psi_{\Gamma,\,\widetilde{P}_\Gamma}} & I(\mathcal{O},\,N)/\widetilde{P}_\Gamma
}
\end{equation*}
\caption{A commutative diagram for surjectivity of $\psi_{\Gamma,\,\widetilde{P}_\Gamma}$}
\label{Figure03}
\end{figure}
\par
Assume that $\Gamma$ induces a form class group of level $N$.
By Definition \ref{induce} and Lemma \ref{well-defined}, 
there exists a subgroup $P$ of $P(\mathcal{O},\,N)$ containing $\widetilde{P}_\Gamma$
for which $\psi_{\Gamma,\,P}$ is bijective. 
We then achieve from the commutative diagram in Figure \ref{Figure04} that
$\psi_{\Gamma,\,\widetilde{P}_\Gamma}$ is injective and 
 $P=\widetilde{P}_\Gamma$. 
\begin{figure}[h]
\begin{equation*}
\xymatrixcolsep{4pc}
\xymatrix{
\mathcal{C}_\Gamma(D_\mathcal{O},\,N) \ar@{->}[rr]_{\psi_{\Gamma,\,P}}^{\sim}
\ar@{->>}[ddr]_{\psi_{\Gamma,\,\widetilde{P}_\Gamma}}
 & & I(\mathcal{O},\,N)/P  \\\\
& I(\mathcal{O},\,N)/\widetilde{P}_\Gamma \ar@{->>}[uur]_{\textrm{natural}} & }
\end{equation*}
\caption{A commutative diagram for injectivity of $\psi_{\Gamma,\,\widetilde{P}_\Gamma}$}
\label{Figure04}
\end{figure}
\par
Conversely, assume that $\psi_{\Gamma,\,\widetilde{P}_\Gamma}$ is injective.
Then $\psi_{\Gamma,\,\widetilde{P}_\Gamma}$ is bijective, and hence $\Gamma$ induces a form class group of level $N$ in the sense of Definition \ref{induce}
with $P=\widetilde{P}_\Gamma$. 
\end{proof}

\section {Form class groups as Galois groups}

In this section, we shall prove that the group $\Gamma_G$ induces a form class group of level $N$.
Rather than using Proposition \ref{theoretical}
which is fundamental but theoretical, we shall develop a criterion for a congruence subgroup $\Gamma$ of $\mathrm{SL}_2(\mathbb{Z})$ to induce a form class group of level $N$
in view of Shimura's theory of canonical models. 
\par
It was first mentioned by Stevenhagen (\cite[$\S$4]{Stevenhagen}) that
\begin{equation}\label{F_Ngeneration}
K_{\mathcal{O},\,\{1+N\mathbb{Z}\}}=
K\left(h(\tau_\mathcal{O})~|~h\in\mathcal{F}_N~\textrm{is finite at}~\tau_\mathcal{O}\right).
\end{equation}
See also \cite[Theorem 4]{Cho}. 
By Proposition \ref{Gamma_1(N)}, $\mathcal{C}_{\Gamma_1(N)}(D_\mathcal{O},\,N)$ is 
a form class group of level $N$. 
Let
\begin{equation*}
\phi_{\mathcal{O},\,\{1+N\mathbb{Z}\}}:\mathcal{C}_{\Gamma_1(N)}(D_\mathcal{O},\,N)
\stackrel{\sim}{\rightarrow}\mathrm{Gal}(K_{\mathcal{O},\,\{1+N\mathbb{Z}\}}/K)
\end{equation*}
be the isomorphism obtained by composing the following three isomorphisms
\begin{enumerate}
\item[(i)] $\psi_{\Gamma_1(N),\,P_{\{1+N\mathbb{Z}\}}(\mathcal{O},\,N)}:\mathcal{C}_{\Gamma_1(N)}(D_\mathcal{O},\,N)
\stackrel{\sim}{\rightarrow}I(\mathcal{O},\,N)/P_{\{1+N\mathbb{Z}\}}(\mathcal{O},\,N)$
sending $[Q]$ to $[\omega_Q,\,1]$
stated in Proposition \ref{Gamma_1(N)}, 
\item[(ii)] $I(\mathcal{O},\,N)/P_{\{1+N\mathbb{Z}\}}(\mathcal{O},\,N)\stackrel{\sim}{\rightarrow}
I(\mathcal{O}_K,\,\ell_\mathcal{O}N)/
P_G(\mathcal{O}_K,\,\ell_\mathcal{O}N,\,\ell_\mathcal{O}N)$ given in Corollary \ref{CK},
\item[(iii)] $I(\mathcal{O}_K,\,\ell_\mathcal{O}N)/
P_G(\mathcal{O}_K,\,\ell_\mathcal{O}N,\,\ell_\mathcal{O}N)\stackrel{\sim}{\rightarrow}\mathrm{Gal}(K_{\mathcal{O},\,\{1+N\mathbb{Z}\}}/K)$
induced by the Artin map for the modulus $\ell_\mathcal{O}N\mathcal{O}_K$. 
\end{enumerate}

\begin{proposition}\label{CG}
We have an explicit description of $\phi_{\mathcal{O},\,\{1+N\mathbb{Z}\}}$ as
\begin{eqnarray*}
\phi_{\mathcal{O},\,\{1+N\mathbb{Z}\}}~:~\mathcal{C}_{\Gamma_1(N)}(D_\mathcal{O},\,N)
& \stackrel{\sim}{\rightarrow} & \mathrm{Gal}(K_{\mathcal{O},\,\{1+N\mathbb{Z}\}}/K)\\
\mathrm{[}Q] & \mapsto &
\left(
h(\tau_\mathcal{O})\mapsto h^{\left[\begin{smallmatrix}
1 & -a'(\frac{b+b_\mathcal{O}}{2})\\
0& a'\end{smallmatrix}\right]}(-\overline{\omega}_Q)~|~h\in\mathcal{F}_N~\textrm{is finite at}~\tau_\mathcal{O}\right)
\end{eqnarray*}
where $Q=ax^2+bxy+cy^2\in\mathcal{Q}(D_\mathcal{O},\,N)$ and $a'$ is an integer such that $aa'\equiv1\Mod{N}$. 
\end{proposition}
\begin{proof}
See \cite[Theorem 12.3]{J-K-S-Y23}.
\end{proof}

\begin{lemma}\label{qbar}
Let $h$ be a meromorphic modular function for $\Gamma_1(N)$ with rational Fourier coefficients. 
If $h$ is finite at a point $z$ in $\mathbb{H}$, then it is also finite at $-\overline{z}$ and
satisfies $h(-\overline{z})=\overline{h(z)}$. 
\end{lemma}
\begin{proof}
It is immediate.
\end{proof}

\begin{proposition}\label{criterion}
Let $\Gamma$ be a subgroup of $\mathrm{SL}_2(\mathbb{Z})$ containing $\Gamma_1(N)$. 
If there is an open subgroup $S$ of $\mathrm{GL}_2(\widehat{\mathbb{Q}})$ containing 
$\mathbb{Q}^\times$ such that $S/\mathbb{Q}^\times$ is compact and
\begin{enumerate}
\item[\textup{(i)}] $\Gamma_S/\mathbb{Q}^\times\simeq\langle\Gamma,\,-I_2\rangle/\langle-I_2\rangle$ in the sense of \textup{Proposition \ref{canonical} (i)}, 
\item[\textup{(ii)}] $S\supseteq\left\{\begin{bmatrix}1&0\\0&d\end{bmatrix}~|~d\in\widehat{\mathbb{Z}}^\times\right\}$,
\end{enumerate}
then  $\Gamma$ induces a form class group of level $N$. 
\end{proposition}
\begin{proof}
By Propositions \ref{generation} and \ref{CG},
we also have the isomorphism
\begin{eqnarray*}
\phi_{\mathcal{O},\,\{1+N\mathbb{Z}\}}~:~\mathcal{C}_{\Gamma_1(N)}(D_\mathcal{O},\,N)&\stackrel{\sim}{\rightarrow}&\mathrm{Gal}(K_{\mathcal{O},\,\{1+N\mathbb{Z}\}}/K)\\
\mathrm{[}Q] & \mapsto & \left(
h(\tau_\mathcal{O})\mapsto h(-\overline{\omega}_Q)~|~h\in\mathcal{F}_{\Gamma_1(N),\,\mathbb{Q}}~\textrm{is finite at}~\tau_\mathcal{O}\right).
\end{eqnarray*}
Observe that if $h\in\mathcal{F}_{\Gamma_1(N),\,\mathbb{Q}}$ is finite at $\tau_\mathcal{O}$, 
then it is finite as well at $-\overline{\omega}_Q$ for every $Q\in\mathcal{Q}(D_\mathcal{O},\,N)$. 
Let $\mathcal{F}_{\Gamma,\,\mathbb{Q}}$ be
the field of meromorphic modular functions for $\Gamma$ with rational Fourier coefficients, 
and let
\begin{equation*}
L=K\left(h(\tau_\mathcal{O})~|~h\in\mathcal{F}_{\Gamma,\,\mathbb{Q}}~\textrm{is finite at}~\tau_\mathcal{O}\right).
\end{equation*}
Since $\Gamma$ contains $\Gamma_1(N)$, $\mathcal{F}_{\Gamma,\,\mathbb{Q}}$ is a subfield
of $\mathcal{F}_{\Gamma_1(N),\,\mathbb{Q}}$, and so $L$ is a subfield of $K_{\mathcal{O},\,\{1+N\mathbb{Z}\}}$. 
Hence $\phi_{\mathcal{O},\,\{1+N\mathbb{Z}\}}$ yields the surjective homomorphism
\begin{eqnarray*}
\mathcal{C}_{\Gamma_1(N)}(D_\mathcal{O},\,N)&\twoheadrightarrow&\mathrm{Gal}(L/K)\\
\mathrm{[}Q] & \mapsto & \left(
h(\tau_\mathcal{O})\mapsto h(-\overline{\omega}_Q)~|~h\in\mathcal{F}_{\Gamma,\,\mathbb{Q}}~\textrm{is finite at}~\tau_\mathcal{O}\right).
\end{eqnarray*}
Now, consider the map
\begin{eqnarray*}
\phi~:~\mathcal{C}_\Gamma(D_\mathcal{O},\,N)&\rightarrow&\mathrm{Gal}(L/K)\\
\mathrm{[}Q] & \mapsto & \left(
h(\tau_\mathcal{O})\mapsto h(-\overline{\omega}_Q)~|~h\in\mathcal{F}_{\Gamma,\,\mathbb{Q}}~\textrm{is finite at}~\tau_\mathcal{O}\right)
\end{eqnarray*}
which is not necessarily well defined. 
We then find that for $Q,\,Q'\in\mathcal{Q}(D_\mathcal{O},\,N)$ 
\begin{eqnarray*}
&&h(-\overline{\omega}_Q)=h(-\overline{\omega}_{Q'})~\textrm{for all}~
h\in\mathcal{F}_{\Gamma,\,\mathbb{Q}}~\textrm{which are finite at}~\tau_\mathcal{O}\\
&\Longleftrightarrow&\overline{h(-\overline{\omega}_Q)}=
\overline{h(-\overline{\omega}_{Q'})}~\textrm{for all}~
h\in\mathcal{F}_{\Gamma,\,\mathbb{Q}}~\textrm{which are finite at}~\tau_\mathcal{O}\\
&\Longleftrightarrow&h(\omega_Q)=h(\omega_{Q'})~\textrm{for all}~
h\in\mathcal{F}_{\Gamma,\,\mathbb{Q}}~\textrm{which are finite at}~\tau_\mathcal{O}~
\textrm{by Lemma \ref{qbar}}\\
& \Longleftrightarrow &
\omega_Q=\gamma(\omega_{Q'})~
\textrm{for some}~\gamma\in\Gamma~\textrm{by (i), (ii) and Lemma \ref{kSQ}}\\
& \Longleftrightarrow &  
Q'=Q^\gamma\quad\textrm{for some}~\gamma\in\Gamma\\
&\Longleftrightarrow&[Q]=[Q']~\textrm{in}~\mathcal{C}_\Gamma(D_\mathcal{O},\,N).
\end{eqnarray*}
This argument shows that $\phi$ is well defined and is injective. Moreover, 
we deduce by the commutative diagram in Figure \ref{Figure05} that $\phi$ is surjective
and $\phi([Q_0])=\mathrm{id}_L$. 
\begin{figure}[h]
\begin{equation*}
\xymatrixcolsep{10pc}
\xymatrix{
\mathcal{C}_{\Gamma_1(N)}(D_\mathcal{O},\,N) \ar@{->}[r]^{\sim}_{\phi_{\mathcal{O},\,\{1+N\mathbb{Z}\}}}
\ar@{->>}[dd]_{\textrm{natural}}
 & \mathrm{Gal}(K_{\mathcal{O},\,\{1+N\mathbb{Z}\}}/K) \ar@{->>}[dd]^{\textrm{restriction}} \\\\
\mathcal{C}_\Gamma(D_\mathcal{O},\,N) \ar@{->}[r]_{\phi} & \mathrm{Gal}(L/K)
}
\end{equation*}
\caption{A commutative diagram for surjectivity of $\phi$}
\label{Figure05}
\end{figure}
\par
Let $P$ be the subgroup of $I(\mathcal{O},\,N)$ containing
$P_{\{1+N\mathbb{Z}\}}(\mathcal{O},\,N)$ such that the 
image of the subgroup $\mathrm{Gal}(K_{\mathcal{O},\,\{1+N\mathbb{Z}\}}/L)$ of 
$\mathrm{Gal}(K_{\mathcal{O},\,\{1+N\mathbb{Z}\}}/K)$ under
the isomorphism
\begin{equation*}
\psi_{\Gamma_1(N),\,P_{\{1+N\mathbb{Z}\}}(\mathcal{O},\,N)}\circ
\phi_{\mathcal{O},\,\{1+N\mathbb{Z}\}}^{-1}:\mathrm{Gal}(K_{\mathcal{O},\,\{1+N\mathbb{Z}\}}/K) \stackrel{\sim}{\rightarrow}
I(\mathcal{O},\,N)/P_{\{1+N\mathbb{Z}\}}(\mathcal{O},\,N)
\end{equation*}
is $P/P_{\{1+N\mathbb{Z}\}}(\mathcal{O},\,N)$. 
Then, the mapping
\begin{eqnarray*}
\psi~:~\mathcal{C}_\Gamma(D_\mathcal{O},\,N) & \rightarrow & I(\mathcal{O},\,N)/P\\
\mathrm{[}Q] & \mapsto & [[\omega_Q,\,1]]
\end{eqnarray*}
is a well-defined bijection which makes the diagram in Figure \ref{Figure06} commute. 
Furthermore, we establish by (\ref{wtO}) that
\begin{eqnarray*}
P&=&\{[\omega_Q,\,1]~|~\textrm{$Q$ is an element of}~\mathcal{Q}(D_\mathcal{O},\,N)~
\textrm{such that}~Q\sim_\Gamma Q_0\}P_{\{1+N\mathbb{Z}\}}(\mathcal{O},\,N)\\
&=&
\{[\omega_{Q_0^\gamma},\,1]~|~\gamma\in\Gamma~\textrm{satisfies}~Q_0^\gamma\in\mathcal{Q}(D_\mathcal{O},\,N)\}P_{\{1+N\mathbb{Z}\}}(\mathcal{O},\,N)\\
&=&\{[\gamma^{-1}(\omega_{Q_0}),\,1]~|~\gamma\in\Gamma~\textrm{satisfies}~Q_0^\gamma\in\mathcal{Q}(D_\mathcal{O},\,N)\}P_{\{1+N\mathbb{Z}\}}(\mathcal{O},\,N)\\
&=&\{j(\gamma^{-1},\,\omega_{Q_0})^{-1}[\omega_{Q_0},\,1]~|~\gamma\in\Gamma~\textrm{satisfies}~Q_0^\gamma\in\mathcal{Q}(D_\mathcal{O},\,N)\}P_{\{1+N\mathbb{Z}\}}(\mathcal{O},\,N)\\
&=&\{j(\gamma^{-1},\,\tau_\mathcal{O})^{-1}\mathcal{O}~|~\gamma\in\Gamma~\textrm{satisfies}~Q_0^\gamma\in\mathcal{Q}(D_\mathcal{O},\,N)\}P_{\{1+N\mathbb{Z}\}}(\mathcal{O},\,N),
\end{eqnarray*}
which claims that $P$ is a subgroup of $P(\mathcal{O},\,N)$.
Therefore we conclude that 
$\Gamma$ induces a form class group of level $N$. 
\begin{figure}[h]
\begin{equation*}
\xymatrixcolsep{7pc}
\xymatrix{
I(\mathcal{O},\,N)/P_{\{1+N\mathbb{Z}\}}(\mathcal{O},\,N)\ar@{->>}[dd]_{\textrm{natural}} &
\mathcal{C}_{\Gamma_1(N)}(D_\mathcal{O},\,N) \ar@{->}[r]^{\sim}_{\phi_{\mathcal{O},\,\{1+N\mathbb{Z}\}}}
\ar@{->>}[dd]_{\textrm{natural}}\ar@{->}[l]_{\sim}^{~~~~~~~\psi_{\Gamma_1(N),\,P_{\{1+N\mathbb{Z}\}}(\mathcal{O},\,N)}}
 & \mathrm{Gal}(K_{\mathcal{O},\,\{1+N\mathbb{Z}\}}/K) \ar@{->>}[dd]^{\textrm{restriction}} 
\\\\ 
I(\mathcal{O},\,N)/P &
 \mathcal{C}_\Gamma(D_\mathcal{O},\,N) \ar@{->}[r]^{\textrm{bijective}}_{\phi}
\ar@{->}[l]_{\textrm{bijective}}^{\psi} 
 & \mathrm{Gal}(L/K) 
}
\end{equation*}
\caption{A commutative diagram showing that $\mathcal{C}_\Gamma(D_\mathcal{O},\,N)$ is a form class group}
\label{Figure06}
\end{figure}
\end{proof}

\begin{theorem}\label{Gformclassgroup}
The map
\begin{eqnarray*}
\psi_{\Gamma_G,\,P_G(\mathcal{O},\,N)}~:~
\mathcal{C}_{\Gamma_G}(D_\mathcal{O},\,N) &\rightarrow&  I(\mathcal{O},\,N)/P_G(\mathcal{O},\,N)\\
\mathrm{[}Q] &\mapsto& [[\omega_Q,\,1]]
\end{eqnarray*}
is a well-defined bijection, and so the group $\Gamma_G$ induces a form class group of level $N$. 
\end{theorem}
\begin{proof}
The result follows from the proof of Theorem \ref{generation} and Proposition \ref{criterion}.
\end{proof}

Let
\begin{eqnarray*}
k_G&=&\textrm{the fixed field of $\mathbb{Q}(\zeta_N)$ by $\det(G^2)$ $(\leq(\mathbb{Z}/N\mathbb{Z})^\times
\simeq\mathrm{Gal}(\mathbb{Q}(\zeta_N)/\mathbb{Q})$)},\\
\mathcal{F}_{\Gamma_G,\,k_G}&=&\textrm{the field of meromorphic modular functions for $\Gamma_G$}\\
&&\textrm{whose Fourier coefficients belong to $k_G$}.
\end{eqnarray*}

\begin{corollary}\label{phiOG}
We have $K_{\mathcal{O},\,G}=K\left(h(\tau_\mathcal{O})~|~h\in\mathcal{F}_{\Gamma_G,\,k_G}~\textrm{is finite at}~\tau_\mathcal{O}\right)$ and we get an isomorphism
\begin{eqnarray*}
\phi_{\mathcal{O},\,G}~:~\mathcal{C}_{\Gamma_G}(D_\mathcal{O},\,N)&\stackrel{\sim}{\rightarrow}&\mathrm{Gal}(K_{\mathcal{O},\,G}/K)\\
\mathrm{[}Q] & \mapsto & \left(
h(\tau_\mathcal{O})\mapsto h^{\left[\begin{smallmatrix}
1 & -a'(\frac{b+b_\mathcal{O}}{2})\\
0& a'\end{smallmatrix}\right]}(-\overline{\omega}_Q)~|~h\in\mathcal{F}_{\Gamma_G,\,k_G}~\textrm{is finite at}~\tau_\mathcal{O}\right)
\end{eqnarray*}
where $Q=ax^2+bxy+cy^2\in\mathcal{Q}(D_\mathcal{O},\,N)$ and $a'$ is an integer such that $aa'\equiv1\Mod{N}$. 
\end{corollary}
\begin{proof}
Let $L=K\left(h(\tau_\mathcal{O})~|~h\in\mathcal{F}_{G,\,k_G}~\textrm{is finite at}~\tau_\mathcal{O}\right)$. 
By Theorem \ref{generation} and (\ref{F_Ngeneration}), we obtain the inclusions 
\begin{equation*}
K_{\mathcal{O},\,G}\subseteq L
\subseteq K_{\mathcal{O},\,\{1+N\mathbb{Z}\}}.
\end{equation*}
Let $\phi_{\mathcal{O},\,\{1+N\mathbb{Z}\}}:\mathcal{C}_{\Gamma_1(N)}(D_\mathcal{O},\,N)\stackrel{\sim}{\rightarrow}
\mathrm{Gal}(K_{\mathcal{O},\,\{1+N\mathbb{Z}\}}/K)$ be the isomorphism described in Proposition \ref{CG}. 
Now, consider an arbitrary element $\rho$ of $\mathrm{Gal}(K_{\mathcal{O},\,\{1+N\mathbb{Z}\}}/K_{\mathcal{O},\,G})$.
By Proposition \ref{Gamma_1(N)}, Remark \ref{Q0} and Theorem \ref{Gformclassgroup}, we have
$\rho=\phi_{\mathcal{O},\,\{1+N\mathbb{Z}\}}([Q])$ for some $Q=ax^2+bxy+cy^2\in\mathcal{Q}(D_\mathcal{O},\,N)$ such that
\begin{equation}\label{QQ0}
Q=Q_0^\gamma\quad\textrm{for some}~\gamma=\begin{bmatrix}p&q\\r&s\end{bmatrix}\in\Gamma_G. 
\end{equation}
Then, we see from the facts $a=p^2+b_\mathcal{O}pr+c_\mathcal{O}r^2$ and 
\begin{equation*}
\begin{bmatrix}p&q\\r&s\end{bmatrix}\equiv\begin{bmatrix}t&\mathrm{*}\\
0&\mathrm{*}\end{bmatrix}\Mod{NM_2(\mathbb{Z})}~\textrm{for some integer $t$ with $t+N\mathbb{Z}\in G$}
\end{equation*}
that if $a'$ is an integer satisfying $aa'\equiv1\Mod{N}$, then 
\begin{equation}\label{a'}
t^2a'\equiv 1\Mod{N}.
\end{equation}
Thus we derive by Proposition \ref{CG} that 
for any $h\in\mathcal{F}_{\Gamma_G,\,k_G}$ finite at $\tau_\mathcal{O}$
\begin{eqnarray*}
h(\tau_\mathcal{O})^\rho&=&
h(\tau_\mathcal{O})^{\phi_{\mathcal{O},\,\{1+N\mathbb{Z}\}}([Q])}\\
&=&h^{\left[\begin{smallmatrix}1 & -a'(\frac{b+b_\mathcal{O}}{2})\\
0& a'\end{smallmatrix}\right]}(-\overline{\omega}_Q)\\
&=&h^{\left[\begin{smallmatrix}1 & -\frac{b+b_\mathcal{O}}{2}\\
0& 1\end{smallmatrix}\right]\left[\begin{smallmatrix}1 & 0\\
0& a'\end{smallmatrix}\right]}(-\overline{\omega}_Q)\\
&=&h^{\left[\begin{smallmatrix}1 & 0\\
0& a'\end{smallmatrix}\right]}(-\overline{\omega}_Q)\quad\textrm{because $h$ is modular for}~
\Gamma_G \left(\ni\begin{bmatrix}1 & -\frac{b+b_\mathcal{O}}{2}\\0& 1\end{bmatrix}\right)\\
&=&h(-\overline{\omega}_Q)\quad\textrm{by (\ref{a'}) and the fact that
$h$ has Fourier coefficients in $k_G$}\\
&=&h(-\overline{\gamma^{-1}(\omega_{Q_0})})\quad\textrm{by (\ref{QQ0})}\\
&=&h(\alpha(-\overline{\omega}_{Q_0}))\quad\textrm{with}~\alpha=\begin{bmatrix}s&q\\r&p\end{bmatrix}\\
&=&h(-\overline{\omega}_{Q_0})\quad\textrm{since $\alpha\in\Gamma_G$}\\
&=&h\left(\begin{bmatrix}1&b_\mathcal{O}\\0&1\end{bmatrix}
(\omega_{Q_0})\right)\\
&=&h(\tau_\mathcal{O})\quad
\textrm{because}~
\begin{bmatrix}1&b_\mathcal{O}\\0&1\end{bmatrix}\in\Gamma_G~
\textrm{and}~\omega_{Q_0}=\tau_\mathcal{O}.
\end{eqnarray*}
This shows that $\rho$ fixes $L$ elementwise. Therefore we conclude by Galois theory that 
\begin{equation}\label{KLK}
K_{\mathcal{O},\,G}=L=K\left(h(\tau_\mathcal{O})~|~h\in\mathcal{F}_{G,\,k_G}~\textrm{is finite at}~\tau_\mathcal{O}\right). 
\end{equation}
\par
The second part of the corollary follows from Proposition \ref{CG}, (\ref{KLK}) and the commutative diagram in Figure \ref{Figure07}.  
\begin{figure}[t]
\begin{equation*}
\xymatrixcolsep{10pc}
\xymatrix{
\mathcal{C}_{\Gamma_1(N)}(D_\mathcal{O},\,N) \ar@{->}[r]^{\sim}_{\phi_{\mathcal{O},\,\{1+N\mathbb{Z}\}}}
\ar@{->>}[dd]_{\textrm{natural}}
 & \mathrm{Gal}(K_{\mathcal{O},\,\{1+N\mathbb{Z}\}}/K) \ar@{->>}[dd]^{\textrm{restriction}} \\\\
\mathcal{C}_{\Gamma_G}(D_\mathcal{O},\,N) \ar@{->}[r]^{\sim}_{\phi_{\mathcal{O},\,G}} & \mathrm{Gal}(K_{\mathcal{O},\,G}/K)
}
\end{equation*}
\caption{A commutative diagram for $\phi_{\mathcal{O},\,G}$}
\label{Figure07}
\end{figure}
\end{proof}

\section {Primes of the form $x^2+ny^2$ with some additional conditions}

Let $n$ be a positive integer. 
We shall apply the field $K_{\mathcal{O},\,G}$ to the problem of determining primes of the form $x^2+ny^2$ with additional conditions $x+N\mathbb{Z}\in G$ and $y\equiv0\Mod{N}$. 

\begin{lemma}\label{GoverQ}
The field $K_{\mathcal{O},\,G}$ is Galois over $\mathbb{Q}$.
\end{lemma}
\begin{proof}
Here, we use the left action notation for Galois elements. 
If we let 
$\iota$ and 
$\mathfrak{c}$ be the identity and the complex conjugation on $\mathbb{C}$,  respectively,
then we see that 
\begin{equation*}
\iota\rho,~\mathfrak{c}\rho\quad(\rho\in\mathrm{Gal}(K_{\mathcal{O},\,G}/K))
\end{equation*}
are all the distinct embeddings of $K_{\mathcal{O},\,G}$ into $\mathbb{C}$. 
So it suffices to show $\mathfrak{c}(K_{\mathcal{O},\,G})=K_{\mathcal{O},\,G}$
in order to prove that $K_{\mathcal{O},\,G}$ is Galois over $\mathbb{Q}$. 
Note that 
$\mathfrak{c}(K_{\mathcal{O},\,G})$ is a
Galois over $\mathfrak{c}(K)$ and
\begin{equation}\label{GGG}
\mathrm{Gal}(\mathfrak{c}(K_{\mathcal{O},\,G})/\mathfrak{c}(K))=
\mathfrak{c}\,\mathrm{Gal}(K_{\mathcal{O},\,G}/K)\,\mathfrak{c}^{-1}|_{\mathfrak{c}(K_{\mathcal{O},\,G})}
\simeq\mathrm{Gal}(K_{\mathcal{O},\,G}/K).
\end{equation}
Furthermore, since 
\begin{equation}\label{cKK}
\mathfrak{c}(K)=K\quad\textrm{and}\quad\mathfrak{c}(\ell_\mathcal{O}N\mathcal{O}_K)=
\ell_\mathcal{O}N\mathcal{O}_K,
\end{equation}
one can consider the Artin map
\begin{equation*}
\left(\frac{\mathfrak{c}(K_{\mathcal{O},\,G})/\mathfrak{c}(K)}{\cdot}\right):I(\mathcal{O}_K,\,\ell_\mathcal{O}N)
\rightarrow\mathrm{Gal}(\mathfrak{c}(K_{\mathcal{O},\,G})/\mathfrak{c}(K)). 
\end{equation*}
We then find that for $\mathfrak{a}\in I(\mathcal{O}_K,\,\ell_\mathcal{O}N)$
\begin{eqnarray*}
\mathfrak{a}\in\mathrm{ker}
\left(\left(\frac{\mathfrak{c}(K_{\mathcal{O},\,G})/\mathfrak{c}(K)}{\cdot}\right)\right)
&\Longleftrightarrow&
\left(\frac{\mathfrak{c}(K_{\mathcal{O},\,G})/\mathfrak{c}(K)}{\mathfrak{c}(\mathfrak{c}(\mathfrak{a}))}\right)
=\mathrm{id}_{\mathfrak{c}(K_{\mathcal{O},\,G})}
\quad\textrm{because}~\mathfrak{c}\mathfrak{c}=\iota\\
&\Longleftrightarrow&
\mathfrak{c}\left(\frac{K_{\mathcal{O},\,G}/K}{\mathfrak{c}(\mathfrak{a})}\right)\mathfrak{c}^{-1}|_{\mathfrak{c}(K_{\mathcal{O},\,G})}=\mathrm{id}_{\mathfrak{c}(K_{\mathcal{O},\,G})}\\
&\Longleftrightarrow&\left(\frac{K_{\mathcal{O},\,G}/K}{\mathfrak{c}(\mathfrak{a})}\right)
=\mathrm{id}_{K_{\mathcal{O},\,G}}\\
&\Longleftrightarrow&\mathfrak{c}(\mathfrak{a})\in\mathrm{ker}
\left(\left(\frac{K_{\mathcal{O},\,G}/K}{\cdot}\right)\right)\\
&\Longleftrightarrow&\mathfrak{c}(\mathfrak{a})\in
P_G(\mathcal{O}_K,\,\ell_\mathcal{O}N,\,\ell_\mathcal{O}N)\\
&\Longleftrightarrow&\mathfrak{a}=\mathfrak{c}(\mathfrak{c}(\mathfrak{a}))
\in\left\{
\mathfrak{c}(\mathfrak{b})~|~\mathfrak{b}\in 
P_G(\mathcal{O}_K,\,\ell_\mathcal{O}N,\,\ell_\mathcal{O}N)\right\}\\
&\Longleftrightarrow&\mathfrak{a}\in 
P_G(\mathcal{O}_K,\,\ell_\mathcal{O}N,\,\ell_\mathcal{O}N)\\
&&\textrm{by (\ref{cKK}) and the definition of 
$P_G(\mathcal{O}_K,\,\ell_\mathcal{O}N,\,\ell_\mathcal{O}N)$}. 
\end{eqnarray*}
Note from (\ref{GGG}) that 
\begin{equation*}
|\mathrm{Gal}(\mathfrak{c}(K_{\mathcal{O},\,G})/\mathfrak{c}(K))|=|\mathrm{Gal}(K_{\mathcal{O},\,G}/K)|
=|I(\mathcal{O}_K,\,\ell_\mathcal{O}N)/P_G(\mathcal{O}_K,\,\ell_\mathcal{O}N,\,\ell_\mathcal{O}N)|.
\end{equation*}
Hence, the Artin map $\left(\frac{\mathfrak{c}(K_{\mathcal{O},\,G})/\mathfrak{c}(K)}{\cdot}\right)$
yields the isomorphism
\begin{equation*}
I(\mathcal{O}_K,\,\ell_\mathcal{O}N)/P_G(\mathcal{O}_K,\,\ell_\mathcal{O}N,\,\ell_\mathcal{O}N)
\stackrel{\sim}{\rightarrow}\mathrm{Gal}(\mathfrak{c}(K_{\mathcal{O},\,G})/\mathfrak{c}(K))=
\mathrm{Gal}(\mathfrak{c}(K_{\mathcal{O},\,G})/K). 
\end{equation*}
The existence theorem of class field theory
(\cite[Theorem 8.6]{Cox} or \cite[$\S$V.9]{Janusz}) implies that $\mathfrak{c}(K_{\mathcal{O},\,G})=
K_{\mathcal{O},\,G}$, which proves that $K_{\mathcal{O},\,G}$ is Galois over $\mathbb{Q}$. 
\end{proof}

\begin{lemma}\label{equivalence}
In particular, let $K=\mathbb{Q}(\sqrt{-n})$ and $\mathcal{O}=\mathbb{Z}[\sqrt{-n}]$. 
Let $p$ be a prime not dividing $2nN$. Then the followings are equivalent\,\textup{:}
\begin{enumerate}
\item[\textup{(i)}] $p=x^2+ny^2$ for some $x,\,y\in\mathbb{Z}$
such that $x+N\mathbb{Z}\in G$ and $y\equiv0\Mod{N}$. 
\item[\textup{(ii)}] $p\mathcal{O}_K=\mathfrak{p}\overline{\mathfrak{p}}$
for a prime ideal $\mathfrak{p}$ of $\mathcal{O}_K$ for which
$\mathfrak{p}\neq\overline{\mathfrak{p}}$ and $\mathfrak{p}=\nu\mathcal{O}_K$ with
$\nu\in\mathcal{O}$ satisfying   
$\nu\equiv m\Mod{N\mathcal{O}}$ for some $m\in\mathbb{Z}$ such that $m+N\mathbb{Z}\in G$. 
\item[\textup{(iii)}] $p\mathcal{O}_K=\mathfrak{p}\overline{\mathfrak{p}}$
for a prime ideal $\mathfrak{p}$ of $\mathcal{O}_K$ such that
$\mathfrak{p}\neq\overline{\mathfrak{p}}$ and $\mathfrak{p}\in P_G(\mathcal{O}_K,\,\ell_\mathcal{O}N,\,\ell_\mathcal{O}N)$. 
\item[\textup{(iv)}] $p\mathcal{O}_K=\mathfrak{p}\overline{\mathfrak{p}}$
for a prime ideal $\mathfrak{p}$ of $\mathcal{O}_K$ such that
$\mathfrak{p}\neq\overline{\mathfrak{p}}$ and $\left(\frac{K_{\mathcal{O},\,G}/K}{\mathfrak{p}}
\right)=\mathrm{id}_{K_{\mathcal{O},\,G}}$. 
\item[\textup{(v)}] $p\mathcal{O}_K=\mathfrak{p}\overline{\mathfrak{p}}$
for a prime ideal $\mathfrak{p}$ of $\mathcal{O}_K$ such that
$\mathfrak{p}\neq\overline{\mathfrak{p}}$ and $\mathfrak{p}$ splits completely in $K_{\mathcal{O},\,G}$. 
\item[\textup{(vi)}] $p$ splits completely in $K_{\mathcal{O},\,G}$. 
\end{enumerate}
\end{lemma}
\begin{proof}
The discriminant of $\mathcal{O}=\mathbb{Z}[\sqrt{-n}]$ is $-4n$, and so $-4n=\ell_\mathcal{O}^2d_K$. 
Since $p$ does not divide $2nN$, it
is unramified in $K$. 
\par
Assume that (i) holds, and hence $p=(x+\sqrt{-n}y)(x-\sqrt{-n}y)$.  
If we set $\nu=x+\sqrt{-n}y$ and $\mathfrak{p}=\nu\mathcal{O}_K$, then 
we have $p\mathcal{O}_K=\mathfrak{p}\overline{\mathfrak{p}}$ as the prime
ideal factorization of $p\mathcal{O}_K$. 
Observe that
$\mathfrak{p}\neq\overline{\mathfrak{p}}$ because $p$ is unramified in $K$. 
Moreover, since $y\equiv0\Mod{N}$, we see that
$\nu\equiv x\Mod{N\mathcal{O}}$. 
Therefore (ii) is true.
\par
Conversely, assume that (ii) holds. 
Then there is a pair of integers $x$ and $y$ which satisfies 
$p\mathcal{O}_K=\mathfrak{p}\overline{\mathfrak{p}}$ with
$\mathfrak{p}=(x+\sqrt{-n}y)\mathcal{O}_K$, and 
$x+\sqrt{-n}y\equiv m\Mod{N\mathcal{O}}$ with $m\in\mathbb{Z}$ such that $m+N\mathbb{Z}\in G$. 
It then follows from the fact $\mathcal{O}_K^\times\cap\mathbb{Q}_{>0}=\{1\}$ that
$p=x^2+ny^2$. Moreover, since $N\mathcal{O}=[N\sqrt{-n},\,N]$, we get that
$x+N\mathbb{Z}=m+N\mathbb{Z}\in G$ and $y\equiv0\Mod{N}$, which yields (i). 
\par
Assume that (ii) holds. Since $p$ is relatively prime to $4nN=-\ell_\mathcal{O}^2d_KN$ and
$p\mathcal{O}_K=\mathfrak{p}\overline{\mathfrak{p}}$ with $\mathfrak{p}=\nu\mathcal{O}_K$,
$\mathfrak{p}$ is prime to $\ell_\mathcal{O}N$. 
Since $\nu\equiv m\Mod{N\mathcal{O}}$ and
$N\mathcal{O}=[N\ell_\mathcal{O}\tau_K,\,N]$, we deduce that
$\nu\equiv m+Nk\Mod{\ell_\mathcal{O}N\mathcal{O}_K}$ for some $k\in\mathbb{Z}$. 
Note that $(m+Nk)+N\mathbb{Z}=m+N\mathbb{Z}\in G$. 
Thus $\mathfrak{p}=\nu\mathcal{O}_K$ belongs to $P_G(\mathcal{O}_K,\,\ell_\mathcal{O}N,\,\ell_\mathcal{O}N)$,
which proves (iii).
\par
Conversely, assume that (iii) holds. Then we have
$\mathfrak{p}=\nu\mathcal{O}_K$ with $\nu\in\mathcal{O}_K$ satisfying
$\nu\equiv a\Mod{\ell_\mathcal{O}N\mathcal{O}_K}$ for some $a\in\mathbb{Z}$ such that $a+N\mathbb{Z}\in G$. 
Then (ii) follows from the fact $\ell_\mathcal{O} N\mathcal{O}_K\subseteq N\mathcal{O}$.
\par
The equivalence of (iii) and (iv) are due to the fact that
the Artin map $\left(\frac{K_{\mathcal{O},\,G}/K}{\cdot}\right):
I(\mathcal{O}_K,\,\ell_\mathcal{O}N)\rightarrow\mathrm{Gal}(K_{\mathcal{O},\,G}/K)$ induces
an isomorphism $I(\mathcal{O}_K,\,\ell_\mathcal{O}N)/P_G(\mathcal{O}_K,\,\ell_\mathcal{O}N,\,\ell_\mathcal{O}N)\xrightarrow{\sim}
\mathrm{Gal}(K_{\mathcal{O},\,G}/K)$.
\par
The equivalence of (iv) and (v) is obtained by the fact that the order of the Artin symbol
 $\left(\frac{K_{\mathcal{O},\,G}/K}{\mathfrak{p}}\right)$ 
in $\mathrm{Gal}(K_{\mathcal{O},\,G}/K)$ 
is the inertia degree of $\mathfrak{p}$ in the field extension $K_{\mathcal{O},\,G}/K$
(\cite[$\S$III.1 and $\S$III.2]{Janusz}). 
\par
The equivalence of (v) and (vi) is derived from the fact that
$K_{\mathcal{O},\,G}$ is Galois over $\mathbb{Q}$ by Lemma \ref{GoverQ}. 
\end{proof}

For  a prime $p$, we let 
\begin{equation*}
\left(\frac{d_K}{p}\right)=\left\{\begin{array}{ll}
\textrm{the Legendre symbol} & \textrm{if $p$ is odd},\\
\textrm{the Kronecker symbol} & \textrm{if $p=2$}. 
\end{array}\right.
\end{equation*}

\begin{lemma}\label{real}
Let $L$ be a finite extension of $K$ which is Galois over $\mathbb{Q}$. 
\begin{enumerate}
\item[\textup{(i)}] Then there is a real algebraic integer $\alpha$ which generates $L$ over $K$. 
\item[\textup{(ii)}] Given $\alpha$ as in \textup{(i)}, let $f(X)\in\mathbb{Z}[X]$ be its minimal polynomial over $K$. If
$p$ is a prime not dividing the discriminant of $f(X)$, then 
\begin{equation*}
\textrm{$p$ splits completely in $L$}~\Longleftrightarrow~
\left(\frac{d_K}{p}\right)=1~\textrm{and}~
f(X)\equiv0\Mod{p}~\textrm{has an integer solution}. 
\end{equation*}
\end{enumerate}
\end{lemma}
\begin{proof}
See \cite[Proposition 5.29]{Cox}. 
\end{proof}

\begin{theorem}\label{x^2+ny^2}
Let $n$ be a positive integer.
Then there is a monic irreducible polynomial $f(X)\in\mathbb{Z}[X]$ for which 
if $p$ is a prime dividing neither $2nN$ nor the discriminant of $f(X)$, then 
\begin{eqnarray*}
&&\textrm{$p=x^2+ny^2$ for some $x,\,y\in\mathbb{Z}$
such that $x+N\mathbb{Z}\in G$ and $y\equiv0\Mod{N}$}\\
&\Longleftrightarrow& 
\left(\frac{-n}{p}\right)=1~\textrm{and $f(X)\equiv0\Mod{p}$ has an integer solution. }
\end{eqnarray*}
\end{theorem}
\begin{proof}
Let $K=\mathbb{Q}(\sqrt{-n})$ and $\mathcal{O}=\mathbb{Z}[\sqrt{-n}]$, and so $-4n=\ell_\mathcal{O}^2d_K$. 
By Lemmas \ref{GoverQ} and \ref{real} (i), there exists a real algebraic integer $\alpha$ such that
$K_{\mathcal{O},\,G}=K(\alpha)$. Let $f(X)\in\mathbb{Z}[X]$ be the minimal polynomial of $\alpha$ over $K$. 
Then we deduce by Lemmas \ref{equivalence} and \ref{real} (ii) that
for a prime dividing neither $2nN$ nor the discriminant of $f(X)$
\begin{eqnarray*}
&&\textrm{$p=x^2+ny^2$ for some $x,\,y\in\mathbb{Z}$
such that $x+N\mathbb{Z}\in G$ and $y\equiv0\Mod{N}$}\\
&\Longleftrightarrow& 
\textrm{$p$ splits completely in $K_{\mathcal{O},\,G}$}\\
&\Longleftrightarrow& 
\left(\frac{-n}{p}\right)=\left(\frac{d_K}{p}\right)=1~\textrm{and $f(X)\equiv0\Mod{p}$ has an integer solution. }
\end{eqnarray*}
\end{proof}

\section {Examples of minimal polynomials over $\mathbb{Q}$}

Recall by Lemma \ref{GoverQ} that $K_{\mathcal{O},\,G}$ is Galois over $\mathbb{Q}$. 
In this section, we shall construct the definite form class group isomorphic to
$\mathrm{Gal}(K_{\mathcal{O},\,G}/\mathbb{Q})$.
Furthermore, we shall present some examples of 
the minimal polynomial of a primitive generator of $K_{\mathcal{O},\,\{1+N\mathbb{Z}\}}$ over $\mathbb{Q}$. 
\par
Let $\widetilde{\mathfrak{c}}$ be the element of
$\mathrm{Gal}(K_{\mathcal{O},\,G}/\mathbb{Q})$ obtained by restricting 
the complex conjugation $\mathfrak{c}$ to $K_{\mathcal{O},\,G}$.

\begin{lemma}\label{semi}
We have a decomposition
\begin{equation*}
\mathrm{Gal}(K_{\mathcal{O},\,G}/\mathbb{Q})=
\mathrm{Gal}(K_{\mathcal{O},\,G}/K)\rtimes
\langle\,\widetilde{\mathfrak{c}}\,\rangle\quad(\simeq\mathrm{Gal}(K_{\mathcal{O},\,G}/K)\rtimes(\mathbb{Z}/2\mathbb{Z})),
\end{equation*}
where $\widetilde{\mathfrak{c}}$ acts on $\mathrm{Gal}(K_{\mathcal{O},\,G}/\mathbb{Q})$
by conjugation. 
\end{lemma}
\begin{proof}
Since $[\mathrm{Gal}(K_{\mathcal{O},\,G}/\mathbb{Q}):\mathrm{Gal}(K_{\mathcal{O},\,G}/K)]=
[K:\mathbb{Q}]=2$, $\mathrm{Gal}(K_{\mathcal{O},\,G}/K)$ is 
normal in $\mathrm{Gal}(K_{\mathcal{O},\,G}/\mathbb{Q})$.
Moreover, since
$\widetilde{\mathfrak{c}}\in\mathrm{Gal}(K_{\mathcal{O},\,G}/\mathbb{Q})
\setminus\mathrm{Gal}(K_{\mathcal{O},\,G}/K)$, 
we get that
\begin{equation*}
\mathrm{Gal}(K_{\mathcal{O},\,G}/\mathbb{Q})=\mathrm{Gal}(K_{\mathcal{O},\,G}/K)\rtimes
\langle\,\widetilde{\mathfrak{c}}\,\rangle\quad
(\simeq\mathrm{Gal}(K_{\mathcal{O},\,G}/K)\rtimes(\mathbb{Z}/2\mathbb{Z})),
\end{equation*}
where $\widetilde{\mathfrak{c}}$ acts on $\mathrm{Gal}(K_{\mathcal{O},\,G}/K)$ by conjugation. 
\end{proof}

Given $Q=ax^2+bxy+cy^2\in\mathcal{Q}(D_\mathcal{O},\,N)$ 
we mean by $-Q=(-1)Q$ the negative definite form $-ax^2-bxy-cy^2$.
The group 
$\Gamma_G$ acts on the set 
\begin{equation*}
\mathcal{Q}^\pm(D_\mathcal{O},\,N)=\left\{Q,\,-Q~|~Q\in\mathcal{Q}(D_\mathcal{O},\,N)\right\}
\end{equation*}
from the right as
\begin{equation*}
Q\left(\begin{bmatrix}x\\y\end{bmatrix}\right)^\gamma=Q\left(\gamma\begin{bmatrix}x\\y\end{bmatrix}\right)\quad(Q\in\mathcal{Q}^\pm(D_\mathcal{O},\,N),~\gamma\in\Gamma_G).
\end{equation*}
We then obtain an equivalence relation $\sim^\pm_{\Gamma_G}$ as follows\,:
for $Q,\,Q'\in\mathcal{Q}^\pm(D_\mathcal{O},\,N)$
\begin{equation*}
Q\sim^\pm_{\Gamma_ G} Q'\quad\Longleftrightarrow\quad Q'=Q^\gamma~\textrm{for some}~\gamma\in\Gamma_G. 
\end{equation*}
Let $\mathcal{C}^\pm_{\Gamma_G}(D_\mathcal{O},\,N)=
\mathcal{Q}^\pm(D_\mathcal{O},\,N)/\sim^\pm_{\Gamma_G}$
be the set of equivalence classes. 

\begin{proposition}\label{definite}
One can give the set $\mathcal{C}^\pm_{\Gamma_G}(D_\mathcal{O},\,N)$ a group structure so that
the form class group 
$\mathcal{C}_{\Gamma_G}(D_\mathcal{O},\,N)$ becomes its subgroup and $[-Q_0]$ corresponds to
$\widetilde{\mathfrak{c}}$. 
\end{proposition}
\begin{proof}
Let $\phi=\phi_{\mathcal{O},\,G}:\mathcal{C}_{\Gamma_G}(D_\mathcal{O},\,N)
\stackrel{\sim}{\rightarrow}\mathrm{Gal}(K_{\mathcal{O}, \,G}/K) $ be the isomorphism stated in Corollary \ref{phiOG},
and consider the map
\begin{eqnarray*}
\phi^\pm~:~
\mathcal{C}^\pm_{\Gamma_G}(D_\mathcal{O},\,N) & \rightarrow & \mathrm{Gal}(K_{\mathcal{O}, \,G}/\mathbb{Q}) \\
\mathrm{[}Q] & \mapsto & \phi([\mathrm{sgn}(Q)Q])\,\widetilde{\mathfrak{c}}^{\,\frac{1-\mathrm{sgn}(Q)}{2}}
\end{eqnarray*}
where
\begin{equation*}
\mathrm{sgn}(Q)=\left\{\begin{array}{rl}
1 & \textrm{if $Q$ is positive definite},\\
-1 & \textrm{if $Q$ is negative definite}.
\end{array}\right.
\end{equation*}
Since the action of $\Gamma_G$ on 
$\mathcal{Q}^\pm(D_\mathcal{O},\,N)$ preserves definiteness,  we see that 
$\phi^\pm$ is well defined, 
$\mathcal{C}_{\Gamma_G}(D_\mathcal{O},\,N)$ is a subset of 
$\mathcal{C}^\pm_{\Gamma_G}(D_\mathcal{O},\,N)$
and
\begin{equation*}
|\mathcal{C}^\pm_{\Gamma_G}(D_\mathcal{O},\,N)|=
2|\mathcal{C}_{\Gamma_G}(D_\mathcal{O},\,N)|=|\mathrm{Gal}(K_{\mathcal{O},\,G}/\mathbb{Q})|.
\end{equation*}
 Furthermore, since 
$\mathrm{Gal}(K_{\mathcal{O},\,G}/\mathbb{Q})=
\mathrm{Gal}(K_{\mathcal{O},\,G}/K)\rtimes
\langle\,\widetilde{\mathfrak{c}}\,\rangle$ by Lemma \ref{semi}, we conclude that $\phi^\pm$ is 
bijective. This proves the proposition. 
\end{proof}

For an index vector $\mathbf{v}=\begin{bmatrix}v_1 & v_2\end{bmatrix}\in M_{1,\,2}(\mathbb{Q})\setminus
M_{1,\,2}(\mathbb{Z})$, the Siegel function $g_\mathbf{v}$ defined on $\mathbb{H}$ is given by
the infinite product expansion 
\begin{equation}\label{Siegel}
g_\mathbf{v}(\tau)=-q^{\frac{1}{2}\mathbf{B}_2(v_1)}e^{\pi\mathrm{i}v_2(v_1-1)}
(1-q_z)\prod_{n=1}^\infty(1-q^nq_z)(1-q^nq_z^{-1})\quad(\tau\in\mathbb{H})
\end{equation}
where $q=e^{2\pi\mathrm{i}\tau}$, $q_z=e^{2\pi\mathrm{i}z}$ with $z=v_1\tau+v_2$ and
$\mathbf{B}_2(x)=x^2-x+\frac{1}{6}$ is the second Bernoulli polynomial. 
Note that $g_\mathbf{v}$ has neither a zero nor a pole on $\mathbb{H}$. 
One can find in \cite[p. 29]{K-L} the original definition of $g_\mathbf{v}$ which is defined as the product of
a Klein form and the square of the Dedekind eta function.  

\begin{lemma}\label{Siegelarithmetic}
Let $N\geq2$ and $\mathbf{u},\,\mathbf{v}\in\frac{1}{N}M_{1,\,2}(\mathbb{Z})\setminus M_{1,\,2}(\mathbb{Z})$. 
\begin{enumerate}
\item[\textup{(i)}] If $\mathbf{u}\equiv\mathbf{v}\Mod{M_{1,\,2}(\mathbb{Z})}$ or
$\mathbf{u}\equiv-\mathbf{v}\Mod{M_{1,\,2}(\mathbb{Z})}$, then
$g_\mathbf{u}^{\frac{12N}{\gcd(6,\,N)}}=g_\mathbf{v}^{\frac{12N}{\gcd(6,\,N)}}$. 
\item[\textup{(ii)}] The function $g_\mathbf{v}^{\frac{12N}{\gcd(6,\,N)}}$ belongs to $\mathcal{F}_N$ and
\begin{equation*}
\left(g_\mathbf{v}^{\frac{12N}{\gcd(6,\,N)}}\right)^\gamma=
g_\mathbf{v\gamma}^{\frac{12N}{\gcd(6,\,N)}}\quad(\gamma\in\mathrm{GL}_2(\mathbb{Z})/\langle-I_2\rangle
\simeq\mathrm{Gal}(\mathcal{F}_N/\mathcal{F}_1)). 
\end{equation*}
\end{enumerate}
\end{lemma}
\begin{proof}
See (\ref{Siegel}) and \cite[Theorem 1.1 in Chapter 2 and Lemma 5.1 in Chapter 3]{K-L}.
\end{proof}

\begin{lemma}\label{Siegelgenerate}
Assume that  $D_\mathcal{O}\neq-3,\,-4$ and $N\geq2$, and let
$g=g_{\left[\begin{smallmatrix}0&\frac{1}{N}\end{smallmatrix}\right]}(\tau_\mathcal{O})^\frac{12N}{\gcd(6,\,N)}$.
\begin{enumerate}
\item[\textup{(i)}] The value $g$ is a nonzero real algebraic integer. 
\item[\textup{(ii)}] For any nonzero integer $m$, 
the value $g^m$ generates $K_{\mathcal{O},\,\{1+N\mathbb{Z}\}}$ over $K$.
\end{enumerate}
\end{lemma}
\begin{proof}
\begin{enumerate}
\item[(i)] See \cite[$\S$3]{K-S}, Lemma \ref{singularj} and (\ref{Siegel}).
\item[(ii)] See \cite[Theorem 1.1]{J-K}.
\end{enumerate}
\end{proof}

\begin{lemma}\label{splitcompletely}
Let $L$ be a finite Galois extension of $\mathbb{Q}$.
Let $\nu$ be a primitive generator of $L$ over $\mathbb{Q}$ as an algebraic integer 
with $F(X)=\mathrm{irr}(\nu,\,\mathbb{Q})$ \textup{(}$\in\mathbb{Z}[X]$\textup{)}. 
If $p$ is a prime not dividing the discriminant of $F(X)$, then 
\begin{equation*}
\textrm{$p$ splits completely in $L$}\quad\Longleftrightarrow\quad
\textrm{$F(X)\equiv0\Mod{p}$ has an integer solution}. 
\end{equation*}
\end{lemma}
\begin{proof}
See \cite[Proposition 5.11 (iii)]{Cox}. 
\end{proof}

\begin{example}
Let $K=\mathbb{Q}(\sqrt{-n})$ and $\mathcal{O}=\mathbb{Z}[\sqrt{-n}]$
for a positive integer $n$. 
Assume that $D_\mathcal{O}\neq-3,\,-4$ and $N\geq2$.
If we let $g=g_{\left[\begin{smallmatrix}0&\frac{1}{N}\end{smallmatrix}\right]}(\tau_\mathcal{O})^\frac{12N}{\gcd(6,\,N)}$, then we have
\begin{eqnarray*}
K_{\mathcal{O},\,\{1+N\mathbb{Z}\}}
&=&\mathbb{Q}(\sqrt{d_K},\,g^n)\quad\textrm{for any nonzero integer $n$ by Lemma \ref{Siegelgenerate} (ii)}\\
&=&\mathbb{Q}(\sqrt{d_K}g)\quad\textrm{because $g\in\mathbb{R}$ by Lemma \ref{Siegelgenerate} (i)}.
\end{eqnarray*}
Let $F(X)=\mathrm{irr}(\sqrt{d_K}g,\,\mathbb{Q})$. 
We derive
by Lemmas \ref{equivalence} and \ref{splitcompletely}
that if $p$ is a prime dividing neither $2nN$ nor the discriminant of $F(X)$, then
\begin{eqnarray*}
&&p=x^2+ny^2~\textrm{for some}~x,\,y\in\mathbb{Z}~
\textrm{such that}~x\equiv1\Mod{N}~\textrm{and}~y\equiv0\Mod{N}\\
&\Longleftrightarrow&
F(X)\equiv0\Mod{p}~\textrm{has an integer solution}. 
\end{eqnarray*}
By utilizing Proposition \ref{definite} and Lemma \ref{Siegelarithmetic}, 
one can find several concrete examples of $F(X)$ as shown in Table \ref{table}. 
\end{example}

\section {An analogue of Kronecker's congruence relation}\label{Kro}

In this last section, we shall derive 
an analogue of Kronecker's congruence relation on 
special values of a modular function of higher level
as an application of our form class groups. 
\par
Let
$\sigma_{\mathcal{O},\,G}:\mathcal{C}_G(\mathcal{O},\,N)
\stackrel{\sim}{\rightarrow}\mathrm{Gal}(K_{\mathcal{O},\,G}/K)$ be the isomorphism obtained by composing two isomorphisms
\begin{enumerate}
\item[(i)] $\mathcal{C}_G(\mathcal{O},\,N)\stackrel{\sim}{\rightarrow}
I(\mathcal{O}_K,\,\ell_\mathcal{O}N)/
P_G(\mathcal{O}_K,\,\ell_\mathcal{O}N,\,\ell_\mathcal{O}N)$ given in Corollary \ref{CK},
\item[(ii)] $I(\mathcal{O}_K,\,\ell_\mathcal{O}N)/
P_G(\mathcal{O}_K,\,\ell_\mathcal{O}N,\,\ell_\mathcal{O}N)\stackrel{\sim}{\rightarrow}
\mathrm{Gal}(K_{\mathcal{O},\,G}/K)$
induced by the Artin map for the modulus $\ell_\mathcal{O}N\mathcal{O}_K$. 
\end{enumerate}

\begin{lemma}\label{reciprocity}
Let $s$ and $t$ be integers such that $(s\tau_\mathcal{O}+t)\mathcal{O}$ 
is a nontrivial ideal of $\mathcal{O}$ which is prime to $N$. 
If $f\in\mathcal{F}_N$ is finite at $\tau_\mathcal{O}$, then 
\begin{equation*}
f(\tau_\mathcal{O})^{\sigma_{\mathcal{O},\,\{1+N\mathbb{Z}\}}([(s\tau_\mathcal{O}+t)\mathcal{O}])}=f^{\left[\begin{smallmatrix}
t-b_\mathcal{O}s & -c_\mathcal{O}s\\s&t
\end{smallmatrix}\right]}(\tau_\mathcal{O}).
\end{equation*}
\end{lemma}
\begin{proof}
See (\ref{F_Ngeneration}) and \cite[(3.4)]{Stevenhagen}. 
\end{proof}

\begin{lemma}\label{decomposed}
Let $p$ be a prime not dividing $D_\mathcal{O}$. Then, $p$ is decomposed with respect to $\mathcal{O}$ if and only if
the quadratic congruence $x^2\equiv D_\mathcal{O}\Mod{4p}$ has an integer solution $x=s$. In this case, 
$\mathfrak{p}=\displaystyle\left[\frac{-s+\sqrt{D_\mathcal{O}}}{2},\,p\right]$ is a proper $\mathcal{O}$-ideal satisfying  $p\mathcal{O}=\mathfrak{p}\overline{\mathfrak{p}}$ with $\mathfrak{p}\neq
\overline{\mathfrak{p}}$. 
\end{lemma}
\begin{proof}
See \cite[Theorem 3 in $\S$9.5]{A-G}. 
\end{proof}

Let $j$ be the elliptic modular function defined on $\mathbb{H}$. 

\begin{lemma}\label{singularj}
If $\tau\in K\cap\mathbb{H}$, then the singular value $j(\tau)$ is an algebraic integer. 
\end{lemma}
\begin{proof}
See \cite[Theorem 4 in Chapter 5]{Lang87}. 
\end{proof}

\begin{theorem}\label{congruence}
Let $f$ be a meromorphic modular function for $\Gamma_G$ with rational Fourier coefficients which is
integral over $\mathbb{Z}[j]$. 
If $p$ is a prime such that
\begin{enumerate}
\item[\textup{(i)}]  it is relatively prime to $D_\mathcal{O}N$,
\item[\textup{(ii)}] it is decomposed with respect to $\mathcal{O}$,
\item[\textup{(iii)}] $p+N\mathbb{Z}\in G$ or $-p+N\mathbb{Z}\in G$,
\end{enumerate}
then we have the congruence relation 
\begin{equation*}
\left(f(\omega)^p-f\left(\frac{\omega}{p}\right)\right)
\left(f(\omega)-f\left(\frac{\omega}{p}\right)^p\right)\equiv0\Mod{p\mathcal{O}_{K_{\mathcal{O},\,G}}}~
\textrm{with}~\omega=\frac{s+\sqrt{D_\mathcal{O}}}{2}
\end{equation*}
where $s$ is an integer satisfying $s^2\equiv D_\mathcal{O}\Mod{4p}$. 
\end{theorem}
\begin{proof}
Note by (i), (ii) and Lemma \ref{decomposed} that 
there exists an integer $s$ such that $s^2\equiv D_\mathcal{O}\Mod{4p}$ and
$\mathfrak{p}=\left[\frac{-s+\sqrt{D_\mathcal{O}}}{2} ,\,p\right]$ is a proper $\mathcal{O}$-ideal satisfying
$p\mathcal{O}=\mathfrak{p}\overline{\mathfrak{p}}$ with $\mathfrak{p}\neq\overline{\mathfrak{p}}$. 
Then we see that
\begin{equation}\label{ppO}
\mathfrak{p}=p\mathcal{O}[\omega_Q,\,1]\quad\textrm{where}~Q=px^2+sxy+
\frac{s^2-D_\mathcal{O}}{4p}y^2~(\in\mathcal{Q}(D_\mathcal{O},\,N)).
\end{equation}
Since $j(\omega)$ and $j(\frac{\omega}{p})$ are algebraic integers by Lemma
\ref{singularj} and
$f$ is integral over $\mathbb{Z}[j]$,
$f(\omega)$ and $f(\frac{\omega}{p})$ are also algebraic integers. 
Furthermore, we find that
\begin{equation*}
f(\tau_\mathcal{O})=f\left(\omega-\frac{s+b_\mathcal{O}}{2}\right)
=f\left(\begin{bmatrix}1 & -\frac{s+b_\mathcal{O}}{2}\\0&1\end{bmatrix}(\omega)
\right)=f(\omega)
\end{equation*}
because $f$ is modular for $\Gamma_G$, which shows that $f(\omega)$ lies in $K_{\mathcal{O},\,G}$
($\subseteq K_{\mathcal{O},\,\{1+N\mathbb{Z}\}}$)
 by Theorem \ref{generation}. 
 Now, let $\sigma=\sigma_{\mathcal{O},\,\{1+N\mathbb{Z}\}}$ for convenience. We then derive that
\begin{eqnarray*}
f(\tau_\mathcal{O})^{\sigma([\mathfrak{p}])}
&=&f(\tau_\mathcal{O})^
{\sigma([p\mathcal{O}])\sigma([[\omega_Q,\,1]])}
\quad\textrm{by (\ref{ppO})}\\
&=&f^{\left[\begin{smallmatrix}
p & 0\\0&p
\end{smallmatrix}\right]}(\tau_\mathcal{O})^{\sigma([[\omega_Q,\,1]])}\quad
\textrm{by Lemma \ref{reciprocity}}\\
&=&f^{\left[\begin{smallmatrix}
p & 0\\0&p
\end{smallmatrix}\right]\left[\begin{smallmatrix}1&-p'\left(\frac{s+b_\mathcal{O}}{2}\right)\\0&p'\end{smallmatrix}\right]}(-\overline{\omega}_Q)\quad\textrm{by Proposition \ref{CG}}\\
&&\hspace{3cm}\textrm{where $p'$ is an integer such that $pp'\equiv1\Mod{N}$}\\
&=&f^{\left[\begin{smallmatrix}p&-\frac{s+b_\mathcal{O}}{2}\\0&1\end{smallmatrix}\right]}
\left(\frac{\omega}{p}\right)\\
&=&f^{\left[\begin{smallmatrix}1&0\\0&p\end{smallmatrix}\right]
\left[\begin{smallmatrix}p&-\frac{s+b_\mathcal{O}}{2}\\0&p'\end{smallmatrix}\right]}\left(\frac{\omega}{p}\right)\\
&=&f^{\left[\begin{smallmatrix}p&-\frac{s+b_\mathcal{O}}{2}\\0&p'\end{smallmatrix}\right]}\left(\frac{\omega}{p}\right)\quad\textrm{since $f$ has rational Fourier coefficients}\\
&=&f\left(\frac{\omega}{p}\right)\quad\textrm{by \textup{(iii)} and the fact that $f$ is modular for $\Gamma_G$}. 
\end{eqnarray*}
Thus we have
\begin{equation}\label{f(omega/p)}
f(\omega)^{\sigma([\mathfrak{p}])}=f\left(\frac{\omega}{p}\right).
\end{equation}
On the other hand, we see by (iii) that
$p\mathcal{O}=(-p)\mathcal{O}\in P_G(\mathcal{O},\,N)$, and hence
$\sigma([p\mathcal{O}])|_{K_{\mathcal{O},\,G}}=\mathrm{id}_{K_{\mathcal{O},\,G}}$.
So we get by (\ref{f(omega/p)}) that
\begin{equation}\label{f(omega)}
f(\omega)=f(\omega)^{\sigma([p\mathcal{O}])}=
f(\omega)^{\sigma([\mathfrak{p}])\sigma([\overline{\mathfrak{p}}])}
=f\left(\frac{\omega}{p}\right)^{\sigma([\overline{\mathfrak{p}}])}.
\end{equation}
Let $\mathfrak{q}=\mathfrak{p}\mathcal{O}_K$. 
Since $\mathfrak{q}\overline{\mathfrak{q}}=\mathfrak{p}\overline{\mathfrak{p}}\mathcal{O}_K=
p\mathcal{O}_K$, we derive by (i), (ii) and Lemma \ref{isomorphism} that $\mathfrak{q}$ and $\overline{\mathfrak{q}}$ are 
distinct prime ideals of $\mathcal{O}_K$ which are prime to $\ell_\mathcal{O}N$. 
Note further that if $\mathfrak{P}$ is a prime ideal of $\mathcal{O}_{K_{\mathcal{O},\,G}}$ lying above $p$, then 
so is 
$\overline{\mathfrak{P}}$ by Lemma \ref{GoverQ} and 
$\mathfrak{P}\neq\overline{\mathfrak{P}}$.
Now that
\begin{equation*}
f(\omega)^{\sigma([\mathfrak{p}])}
\equiv f(\omega)^{\left(\frac{K_{\mathcal{O},\,G}/K}{\mathfrak{q}}\right)}
\equiv f(\omega)^{|\mathcal{O}_K/\mathfrak{q}|}
\equiv
 f(\omega)^p\Mod{\mathfrak{P}}
\end{equation*}
and
\begin{equation*}
f\left(\frac{\omega}{p}\right)^{\sigma(\overline{\mathfrak{p}})}
\equiv f\left(\frac{\omega}{p}\right)^{\left(\frac{K_{\mathcal{O},\,G}/K}{\overline{\mathfrak{q}}}\right)}
\equiv 
f\left(\frac{\omega}{p}\right)^{|\mathcal{O}_K/\overline{\mathfrak{q}}|}\equiv
f\left(\frac{\omega}{p}\right)^p\Mod{\overline{\mathfrak{P}}}
\end{equation*}
by the definition of an Artin symbol, we obtain by (\ref{f(omega/p)}) and (\ref{f(omega)}) that 
\begin{equation*}
\left(f(\omega)^p-f\left(\frac{\omega}{p}\right)\right)
\left(f(\omega)-f\left(\frac{\omega}{p}\right)^p\right)\equiv0\Mod{\mathfrak{P}\overline{\mathfrak{P}}}.
\end{equation*}
Lastly, since
\begin{equation*}
p\mathcal{O}_{K_{\mathcal{O},\,G}}=(\mathfrak{q}\mathcal{O}_{K_{\mathcal{O},\,G}})
(\overline{\mathfrak{q}}\mathcal{O}_{K_{\mathcal{O},\,G}})=
(\mathfrak{P}_1\mathfrak{P}_2\cdots\mathfrak{P}_g)
(\overline{\mathfrak{P}}_1\overline{\mathfrak{P}}_2\cdots\overline{\mathfrak{P}}_g),
\end{equation*}
where $\mathfrak{P}_1,\,\mathfrak{P}_2,\,\ldots,\,\mathfrak{P}_g$ 
are prime ideals of $\mathcal{O}_{K_{\mathcal{O},\,G}}$ lying above $\mathfrak{p}$, 
we achieve by the Chinese remainder theorem that 
\begin{equation*}
\left(f(\omega)^p-f\left(\frac{\omega}{p}\right)\right)
\left(f(\omega)-f\left(\frac{\omega}{p}\right)^p\right)\equiv0\Mod{p\mathcal{O}_{K_{\mathcal{O},\,G}}}.
\end{equation*}
\end{proof}

\begin{sidewaystable}
\centering
\caption{Minimal polynomials over $\mathbb{Q}$}\label{table}
{\small
\begin{tabular}{c|c|c|c|c|c}
{\normalsize$K$} & 
{\normalsize$D_\mathcal{O}$} & 
{\normalsize$N$} & 
{\normalsize$\mathcal{C}_{\Gamma_1(N)}(D_\mathcal{O},\,N)$} & 
{\normalsize$F(X)=\mathrm{irr}\left(\sqrt{d_K}g_{\left[\begin{smallmatrix}0&\frac{1}{N}\end{smallmatrix}\right]}(\tau_\mathcal{O})^\frac{12N}{\gcd(6,\,N)},\,\mathbb{Q}\right)\phantom{\bigg|}$} & 
{\normalsize discriminant of $F(X)$} \\
\hline $\mathbb{Q}(\sqrt{-2})$ & $-200$ & $3$ &
$\left\{\begin{array}{l}\textrm{[}x^2+50y^2\textrm{]},\\
\textrm{[}2x^2+25y^2\textrm{]},\\
\textrm{[}17x^2+2xy+3y^2\textrm{]},\\
\textrm{[}17x^2-2xy+3y^2\textrm{]},\\
\textrm{[}11x^2-8xy+6y^2\textrm{]},\\
\textrm{[}11x^2+8xy+6y^2\textrm{]},\\
\textrm{[}50x^2+y^2\textrm{]},\\
\textrm{[}25x^2+2y^2\textrm{]},\\
\textrm{[}22x^2-36xy+17y^2\textrm{]},\\
\textrm{[}22x^2+36xy+17y^2\textrm{]},\\
\textrm{[}25x^2+30xy+11y^2\textrm{]},\\
\textrm{[}25x^2-30xy+11y^2\textrm{]}
\end{array}\right\}$ & 
$\begin{array}{l}
X^{24}
+58418434677344 X^{22}
+1263375231780687917184 X^{20}\\
+403818817043131055680665600 X^{18}\\
+75730968484681312433176242483200 X^{16}\\
+8361096391935757794654559611579531264 X^{14}\\
+860683009678299985386510787472645392695296 X^{12}\\
+9907654477954796832790654933192834007418535936 X^{10}\\
+33315019088321396809058767421430556685071338700800 X^8\\
+19239392992571915645005697694048991576255756867993600 X^6\\
+29212993887308366869993711350192889063288845726933581824 X^4\\
+6293984600086664567543704795614781286383616 X^2\\
+68719476736\end{array}$ &
$\begin{array}{l}
2^{1772}\cdot3^{12}\cdot 5^{68}\cdot 7^{120}\cdot 13^{56}\\
\cdot 23^{56}\cdot29^{32}\cdot 31^{8}\cdot 37^8 \cdot 47^{32}\\
\cdot 53^{12}\cdot 61^8 \cdot71^8 \cdot101^{16} \cdot 149^8\\
\cdot 167^8\cdot 173^{12} \cdot 191^8 \cdot197^4 \\
\cdot 311^4\cdot 431^4 \cdot 719^4 \cdot983^8\\
\cdot1801^4 \cdot7369^4\cdot13679^4 \cdot44449^4\\
\cdot91009^4\cdot104399^4\cdot143567^4\\
\cdot184609^4\cdot255049^4 \cdot482021^4\\
\cdot1521649^4 \cdot3139369^4\cdot3857809^4\\
\cdot8698681^4\cdot260370001^4\cdot272850169^4\\
\cdot404455343^4\cdot1532509721761^4\\
\cdot15630971591656081^4
\end{array}$ \\
\hline $\mathbb{Q}(\sqrt{-5})$ & $-180$ & $2$ & 
$\left\{\begin{array}{l}
\textrm{[}x^2+45y^2\textrm{]},\\
\textrm{[}23x^2-2xy+2y^2\textrm{]},\\
\textrm{[}5x^2+9y^2\textrm{]},\\
\textrm{[}7x^2+4xy+7y^2\textrm{]},\\
\textrm{[}45x^2+y^2\textrm{]},\\
\textrm{[}23x^2-44xy+23y^2\textrm{]},\\
\textrm{[}9x^2+5y^2\textrm{]},\\
\textrm{[}7x^2-4xy+7y^2\textrm{]}
\end{array}\right\}$ & 
$\begin{array}{l}
X^{16}+40370081379856476160 X^{14}\\
-2294213210542224903962053836800 X^{12}\\
+32594776263664443712118696387582885888000 X^{10}\\
+1355997164048299289268149453587358102323200000 X^8\\
-4618215678434035548825390724987200304106700800000 X^6\\
+5341315045070297685630774389596962453603745792000000 X^4\\
+87496192498069022574637171465249162202332528640000000 X^2\\
+7205759403792793600000000\end{array}$ & 
$\begin{array}{l}
2^{1296}\cdot3^8\cdot 5^{180}\cdot 11^{44}\cdot 13^{28}\\
\cdot 17^{32}\cdot19^{32}\cdot 31^{36}\cdot 37^8 \cdot 53^4\\
\cdot71^4\cdot73^4\cdot79^4\cdot97^4\cdot113^4\cdot131^4\\
\cdot137^4\cdot139^4\cdot151^4\cdot157^4\cdot173^4\\
\cdot181^4\cdot229^4\cdot4201^4\cdot5281^4\cdot6911^4\\
\cdot21481^4\cdot39551^4\cdot42709^4\cdot112621^4\\
\cdot117841^4\cdot1567261^4\cdot721400461^4\\
\cdot27666986168641^4\\
\cdot1459141468570561^4
\end{array}$
\end{tabular}}
\end{sidewaystable}

\newpage
\bibliographystyle{amsplain}

\address{
Department of Mathematics\\
Dankook University\\
Cheonan-si, Chungnam 31116\\
Republic of Korea} {hoyunjung@dankook.ac.kr}
\address{
Department of Mathematical Sciences \\
KAIST \\
Daejeon 34141\\
Republic of Korea} {jkgoo@kaist.ac.kr}
\address{
Department of Mathematics\\
Hankuk University of Foreign Studies\\
Yongin-si, Gyeonggi-do 17035\\
Republic of Korea} {dhshin@hufs.ac.kr}
\address{
Department of Mathematics Education\\
Pusan National University\\
Busan 46241\\Republic of Korea}
{dsyoon@pusan.ac.kr}

\end{document}